\documentclass{article}

\usepackage{arxiv}
\usepackage[applemac]{inputenc} 		
\usepackage[T1]{fontenc}    		
\usepackage[colorlinks]{hyperref}      
\usepackage{url}            			
\usepackage{amsthm}
\usepackage{algorithm}
\usepackage{algorithmicx}
\usepackage{algpseudocode}
\usepackage{amsfonts}       		
\usepackage{amsmath}
\usepackage{array}
\usepackage{amssymb}

\usepackage{booktabs}       		
\usepackage{caption}
\usepackage{comment}
\usepackage{dsfont}
\usepackage{enumitem}
\usepackage{framed}
\usepackage{fancyhdr}
\usepackage{graphicx}
\usepackage{indentfirst,latexsym,bm}
\usepackage{lipsum}				
\usepackage{microtype}      		
\usepackage{mathtools}
\usepackage[square,sort,comma,numbers]{natbib}
\usepackage{nicefrac}       		
\usepackage{subcaption}
\usepackage{wrapfig}
\usepackage{epstopdf}
\usepackage{xcolor}         
\usepackage{tikz}
\usepackage{pgfplots}
\usepackage{hyperref}[6.83]
\usepackage{cancel}
\usepackage[title,titletoc]{appendix}
\usetikzlibrary{arrows.meta,positioning}
\pgfplotsset{compat=newest}
\usepgfplotslibrary{fillbetween}
\usetikzlibrary{shapes,decorations}
\usetikzlibrary{fit}

\colorlet{color1}{blue}
\colorlet{color2}{red!50!black}

\definecolor{ivory}{RGB}{218,215,203}

\definecolor{cuhkp}{RGB}{98,56,105} 	
\definecolor{cuhkpl}{RGB}{152,24,147} 	
\definecolor{cuhkb}{RGB}{219,160,1} 	
\definecolor{cuhkbd}{RGB}{178,129,0} 	
\definecolor{cuhkr}{RGB}{88,35,155}  	

\definecolor{darkblue}{rgb}{0,0.08,0.45} 
\hypersetup{
	colorlinks=true,
	frenchlinks=false,
	pdfborder={0 0 0},
	naturalnames=false,
	hypertexnames=false,
	breaklinks,
	linkcolor=darkblue,
	citecolor=darkblue,
	filecolor=darkblue,
	urlcolor = black,
}

\RequirePackage[capitalize,nameinlink]{cleveref}[0.19]

\crefname{section}{section}{sections}
\crefname{subsection}{subsection}{subsections}

\Crefname{figure}{Figure}{Figures}

\crefformat{equation}{\textup{#2(#1)#3}}
\crefrangeformat{equation}{\textup{#3(#1)#4--#5(#2)#6}}
\crefmultiformat{equation}{\textup{#2(#1)#3}}{ and \textup{#2(#1)#3}}
{, \textup{#2(#1)#3}}{, and \textup{#2(#1)#3}}
\crefrangemultiformat{equation}{\textup{#3(#1)#4--#5(#2)#6}}%
{ and \textup{#3(#1)#4--#5(#2)#6}}{, \textup{#3(#1)#4--#5(#2)#6}}{, and \textup{#3(#1)#4--#5(#2)#6}}

\Crefformat{equation}{#2Equation~\textup{(#1)}#3}
\Crefrangeformat{equation}{Equations~\textup{#3(#1)#4--#5(#2)#6}}
\Crefmultiformat{equation}{Equations~\textup{#2(#1)#3}}{ and \textup{#2(#1)#3}}
{, \textup{#2(#1)#3}}{, and \textup{#2(#1)#3}}
\Crefrangemultiformat{equation}{Equations~\textup{#3(#1)#4--#5(#2)#6}}%
{ and \textup{#3(#1)#4--#5(#2)#6}}{, \textup{#3(#1)#4--#5(#2)#6}}{, and \textup{#3(#1)#4--#5(#2)#6}}

\crefdefaultlabelformat{#2\textup{#1}#3}

\theoremstyle{plain}
\newtheorem{theorem}{Theorem}[section]
\newtheorem{lemma}{Lemma}[section]
\newtheorem{corollary}{Corollary}[section]
\newtheorem{proposition}{Proposition}[section]
\newtheorem{assumption}{Assumption}[section]

\newtheorem{fact}{Fact}[section]

\theoremstyle{definition}

\newcommand{\R}{\mathbb{R}}
\newcommand{\N}{\mathbb{N}}

\newcommand{\dist}{\operatorname{dist}}
\newcommand{\sign}{\operatorname{sign}}
\newcommand{\proj}{\operatorname{proj}}

\newcommand{\sB}{{\sf B}}

\newcommand{\be}{\begin{equation}}
\newcommand{\ee}{\end{equation}}

\newcommand\domain[1]{\mathrm{dom\\}(#1)}

\newcommand\diag{\mathrm{diag}}

\newcommand{\calO}{\mathcal{O}}

\newcommand{\setI}{\mathbb{I}}

\newcommand{\setN}{\mathbb{N}}

\newcommand{\setR}{\mathbb{R}}

\newcommand{\setX}{\mathbb{X}}

\newcommand{\en}{\begin{equation*}}
\newcommand{\een}{\end{equation*}}

\title{Randomized Coordinate Subgradient Method \\ for Nonsmooth Composite Optimization\thanks{L. Zhao is also with Xiangfu Laboratory, Jiashan, China. D. Zhu is also with School of Data Science, The Chinese University of Hong Kong, Shenzhen. L. Zhao was partially supported by the National Natural Science Foundation of China (NSFC) under Grant No. 72293580 and 72293582. X. Li  was partially supported by the National Natural Science Foundation of China (NSFC) under Grant No. 12201534 and by Shenzhen Science and Technology Program under Grant No. RCBS20210609103708017.}}

\author{%
  \hspace{-1.3cm} Lei Zhao\\
  \hspace{-1.3cm} Institute of Translational Medicine \\
  \hspace{-1.3cm} National Center for Translational Medicine\\
  \hspace{-1.3cm} Shanghai Jiao Tong University\\
  \hspace{-1.3cm} Shanghai, China 200240 \\
  \hspace{-1.3cm} \texttt{l.zhao@sjtu.edu.cn} \\
  \And
  \hspace{0.3cm} Ding Chen\\
  \hspace{0.3cm} Alibaba Group\\
  \hspace{0.3cm} Hangzhou, China 310000 \\
  \hspace{0.3cm} \texttt{ppposchen@gmail.com} \\
  \AND
  Daoli Zhu\\
  Antai College of Economics and Management \\
  Shanghai Jiao Tong University\\
  Shanghai, China 200030 \\
  \texttt{dlzhu@sjtu.edu.cn} \\
  \And
  Xiao Li\\
  School of Data Science\\
  The Chinese University of Hong Kong, Shenzhen\\
  Shenzhen, China 518172\\
  \texttt{lixiao@cuhk.edu.cn} \\
}

\begin{document}

\maketitle
\begin{abstract}
  Coordinate-type subgradient methods for addressing nonsmooth optimization problems are relatively underexplored due to the set-valued nature of the subdifferential. In this work, our study focuses on nonsmooth composite optimization problems, encompassing a wide class of convex and weakly convex (nonconvex nonsmooth) problems. By utilizing the chain rule of the composite structure properly, we introduce the Randomized Coordinate Subgradient method (RCS) for tackling this problem class. To the best of our knowledge, this is \emph{the first} coordinate subgradient method for solving general nonsmooth composite optimization problems. In theory, we consider the linearly bounded subgradients assumption for the objective function, which is more general than the traditional Lipschitz continuity assumption, to account for practical scenarios. We then conduct convergence analysis for RCS in both convex and weakly convex cases based on this generalized Lipschitz-type assumption. Specifically, we establish the $\widetilde \calO (1/\sqrt{k})$ convergence rate in expectation and the $\tilde o(1/\sqrt{k})$ almost sure asymptotic convergence rate in terms of the suboptimality gap when $f$ is convex, where $k$ represents iteration count. If $f$ further satisfies the global quadratic growth condition, the improved $\calO(1/k)$ rate is derived in terms of the squared distance to the optimal solution set.   For the case when $f$ is weakly convex and its subdifferential satisfies the global metric subregularity property, we derive the $\mathcal{O}(\varepsilon^{-4})$ iteration complexity in expectation, where $\varepsilon$ is the target accuracy. We also establish an asymptotic convergence result. To justify the global metric subregularity property utilized in the analysis,  we establish this error bound condition for the concrete (real-valued) robust phase retrieval problem. We also provide a convergence lemma and the relationship between the global metric subregularity properties of a weakly convex function and its Moreau envelope. These results are of independent interest. Finally, we conduct several experiments to demonstrate the possible superiority of RCS over the subgradient method.
\end{abstract}

\section{Introduction}\label{sec:introduction}
Coordinate-type methods tackle optimization problems by successively solving simpler (even scalar) subproblems along coordinate directions. These methods are widely used in signal processing, machine learning, and other engineering fields due to their simplicity and efficiency, see, e.g.,
\cite{Wright2015,kerahroodi2017coordinate,vandaele2016efficient,Tseng}. Coordinate methods are particularly useful when the problem size (measured as the dimension of the variable) is enormous and the computation of function values or full gradients can exhaust memory, or when the problem data is incrementally received or distributed across a network, requiring computation with whatever data is currently available. For more motivational statements, we refer to \cite{Wright2015, Nesterov, RT}.
\subsection{Motivations, problem setup, and examples}\label{sec:motivation and problem}
Existing works in this field mainly focus on extending gradient-based methods  (e.g., the gradient descent, proximal gradient descent, and primal-dual methods) to coordinate methods; see, e.g., \cite{Nesterov,beck2013convergence,LuXiao,RT,Wright2015,ZhuZhao2020,LiuWright1} and the references therein. In contrast, a parallel line of research on implementing a coordinate-type extension for the \emph{subgradient oracle} is relatively unexplored and the related works in this line are surprisingly limited (to the best of our knowledge). The primary challenge in designing a coordinate-type subgradient method for general nonsmooth optimization problems stems from the \emph{set-valued nature} of the subdifferential. Specifically, the Cartesian product of coordinate-wise subgradients may not necessarily give rise to a complete subgradient, rendering the convergence of such a method unclear in both practical implementation and theoretical analysis. Despite this general difficulty, our goal is to design a coordinate-type subgradient method for a class of \emph{structured} nonsmooth optimization problems.

In this work, we consider the \emph{composite optimization} problem
\begin{equation}\label{eq:opt problem}
\min_{x\in \R^{d}} \ f(x)=h(\Phi(x)),
\end{equation}
where the outer function $h:\R^n \to \R$ is possibly nonsmooth and the inner function $\Phi:\R^d\to \R^n$ is a smooth mapping. In addition, the objective function $f:\R^d\to \R$ is  $\rho$-weakly convex (i.e., $f+\frac{\rho}{2} \|\cdot\|^2$ is convex for some $\rho\geq 0$). When $\rho = 0$, $f$ reduces to a convex function.

The composite form \eqref{eq:opt problem} is a common source of nonsmooth convex and weakly convex functions; see, e.g., \cite{lewis2016proximal,duchi2018stochastic,davis2019stochastic}. It covers many signal processing and machine learning problems as special instances. We list three of them below. One can easily put the three examples into the composite form \eqref{eq:opt problem}, while the weak convexity of $f$ in these examples comes from \emph{one} of the following two cases:
\begin{enumerate}[label=\textup{\textrm{(C.\arabic*)}},topsep=0pt,itemsep=0.2ex,partopsep=0ex, leftmargin = 1.2cm]
\item \label{C1} The outer function $h$ is convex and $L_h$-Lipschitz continuous and the derivative of the inner smooth mapping $\Phi$ (presented by the $n\times d$ Jacobian matrix $\nabla\Phi$) is $L_{\Phi}$-Lipschitz continuous. In this case, $f$ is $L_hL_{\Phi}$-weakly convex according to \cite[Lemma 4.2]{DruPaq19}.
\item \label{C2} The outer function $h$ is $\rho_h$-weakly convex and $L_h$-Lipschitz continuous and the inner smooth mapping $\Phi$ is $L_{\Phi}^0$-Lipschitz continuous and its Jacobian $\nabla \Phi$ is $L_{\Phi}$-Lipschitz continuous. In this case, we can establish that $f$ is  $\left(\rho_h(L_{\Phi}^0)^2+L_hL_{\Phi}\right)$-weakly convex; see~\Cref{lemma:wcvx_composite}.
\end{enumerate}

{\sf Application 1: Robust M-estimators.}
The robust M-estimator problem is defined as \cite{LohWai15}:
\be\label{eq:M-esti}
\min_{x\in\R^d}\ f(x) := \ell(Ax-b)+\sum_{i=1}^d\phi_{p}(x_i),
\ee
where $A\in\R^{n\times d}$ is a matrix, $b\in\R^n$ is a vector, $\ell$ is a Lipschitz continuous loss function like $\ell_1$-norm,  MCP~\cite{MCP} and SCAD~\cite{SCAD}, and $\phi_{p}$ is a Lipschitz continuous penalty function such as $\ell_1$-norm, MCP, SCAD, etc. In this problem, let us define $\Phi: \R^d \to \R^n\times\R^d$ by $\Phi(x)=\left(Ax-b, x\right)$ and $h:\R^n\times\R^d\rightarrow\R$ by $h(y,z)=\ell(y)+\sum_{i=1}^d\phi_{p}(z_i)$. Then, the weak convexity of $f$ follows from the case \ref{C2}. In this way, we have placed problem \eqref{eq:M-esti} in the composite form \eqref{eq:opt problem}.

{\sf Application 2:  SVM.} Support vector machine (SVM) is a popular supervised learning method, which is widely used for classification. The linear SVM problem can be expressed as~\cite{cortes1995, LSVM2}
\be\label{eq:svm}
\min_{x\in\R^d}\ f(x) := \frac{1}{n}\sum_{i=1}^n\max\left\{0,1-b_i(a_i^{\top}x)\right\}+\frac{p}{2}\|x\|^2,
\ee
where $p>0$, $a_i\in\R^d$, and $b_i\in\{\pm1\}$, $i=1,...,n$. Let $A=\left(a_1^\top,...,a_n^\top\right)\in\R^{n\times d}$, $b=\left(b_1,...,b_n\right)\in\R^n$,  $\R^{n\times d}\ni\tilde{A}=\diag(b) A$, and $1_n$ be a $n$-dimensional vector of all $1$s. We define $\Phi:\R^d \to \R^n\times \R$ by $\Phi(x)=(1_n-\tilde{A}x, \frac{p}{2}\|x\|^2)$ and $h:\R^n\times \R \to \R$ by $h(y,z)=\frac{1}{n}\sum_{i=1}^n\max\{0,y_i\} + z$. Furthermore, the weak convexity of $f$ follows from the fact that the problem is even convex (i.e., $\rho =0$). Then, one can see that problem \eqref{eq:svm} has the composite form \eqref{eq:opt problem}.

{\sf Application 3: (Real-valued) Robust phase retrieval problem.}  Phase retrieval is a common computational problem with applications in diverse areas such as physics, imaging science, X-ray crystallography, and signal processing \cite{fienup1982phase}. The (real-valued) robust phase retrieval problem amounts to solving \cite{RobustPhase1}
\be\label{eq:pr}
\min_{x\in\R^d}\ f(x) := \frac{1}{n}\|(Ax)^{\circ2}-b^{\circ2}\|_1,
\ee
where $A\in\R^{n\times d}$, $b\in\R^n$, and $\circ2$ is the Hadamard power. It is obvious that problem \eqref{eq:pr} has the composite form in \eqref{eq:opt problem} by setting $\Phi(x)=(Ax)^{\circ2}-b^{\circ2}$ and $h(\cdot)=\frac{1}{n}\|\cdot\|_1$ and realizing that $f$ is weakly convex due to the case \ref{C1}.

It is interesting to observe that $f$ in problem \eqref{eq:M-esti} is Lipschitz continuous but  can be weakly convex due to SCAD and MCP loss function or penalty.  $f$ in problem \eqref{eq:svm} is convex. However, it is not Lipschitz continuous due to the weight decay term $\|x\|^2$. $f$ in problem \eqref{eq:pr} is simultaneously weakly convex and non-Lipschitz continuous. In this work, we aim at designing a coordinate-type subgradient method for solving the nonsmooth composite optimization problem \eqref{eq:opt problem}. Then,  we derive its convergence theory, which covers the weakly convex and/or non-Lipschitz problems such as Applications 1--3.


\subsection{Main contributions}
Our main contributions are summarized as follows.
\begin{enumerate}[label=\textup{\textrm{(\alph*)}},topsep=0pt,itemsep=0ex,partopsep=0ex]
\item Despite the fundamental bottleneck of designing a coordinate-type subgradient method for general nonsmooth optimization problems, we focus on the structured composite problem \eqref{eq:opt problem}. The key insight lies in that this composite structure admits the chain rule subdifferential calculus, which allows to separate the derivative of the smooth mapping $\Phi$ and the subdifferential of the outer nonsmooth function $h$. Consequently, such a structured subdifferential presents a natural coordinate separation. Based on this insight, we design a Randomized Coordinate Subgradient method (RCS) (see \Cref{alg:RCS}) for solving problem \eqref{eq:opt problem}. To our knowledge, RCS is \emph{the first} coordinate-type subgradient method for tackling general nonsmooth composite optimization problems.

\item Motivated by our examples that the traditional Lipschitz continuity assumption on the objective function $f$ can be stringent, we employ the more general linearly bounded subgradients assumption (see \Cref{assump1}) in our analysis. When $f$ is convex (i.e., $\rho = 0$) and under the linearly bounded subgradients assumption: $\bullet$) We derive the $\widetilde \calO(1/\sqrt{k})$ convergence rate in expectation in terms of the suboptimality gap (see \Cref{theo:rate_cvx}). $\bullet$) We show $\calO(1/k)$  convergence rate in expectation with respect to the squared distance between the iterate and the optimal solution set under an additional global quadratic growth condition (see \Cref{cor:rate_qg}). $\bullet$) We establish the almost sure asymptotic convergence result (see \Cref{theo:convergence_cvx}) and the almost sure $\tilde o(1/\sqrt{k})$  asymptotic convergence rate result (see \Cref{cor:asymptoticrate_cvx}) in terms of the suboptimality gap, which are valid for each single run of RCS with probability 1. One important step for deriving these convergence results is to establish the boundedness of the iterates in expectation so that the technical difficulty introduced by the more general linearly bounded subgradients assumption can be tackled by using diminishing step sizes. We refer to \Cref{sec:cvx} for more details.
	
\item When $f$ is weakly convex (i.e., $\rho>0$) and under the linearly bounded subgradients assumption: $\bullet$) We provide the  iteration complexity of $\calO(\varepsilon^{-4})$ in expectation for driving the Moreau envelope gradient below $\varepsilon$ (see \Cref{theo:rate_wcvx}). $\bullet$) We derive the almost sure asymptotic convergence  result (see \Cref{theo:convergence_wcvx}) and the almost sure $\calO(1/k^{-\frac14})$ asymptotic convergence rate result (in the sense of liminf; see \Cref{cor:asymptoticrate_wcvx}) in terms of the Moreau envelope gradient.  Our analysis for weakly convex case deviates from the standard one \cite{davis2019stochastic} due to the more general \Cref{assump1}. Our key idea is to utilize an error bound condition (the global metric subregularity) to derive the approximate descent property of RCS (see \Cref{lemma:inequality_iter_wcvx}). Towards establishing the desired results, we derive a convergence lemma for mapping (see \Cref{lemma:lemma4in84}) and the relationship between the global metric subregularity properties of a weakly convex function and its Moreau envelope (see \Cref{cor:me_ms}), which are of independent interests.
	
\item We establish that the (real-valued) robust phase retrieval problem \eqref{eq:pr} satisfies the global metric subregularity property under a mild condition (see \Cref{cor:pr}). This finding validates the key assumption we utilized for analyzing RCS in the weakly convex case for this specific problem. This established error bound condition is expected to be useful for analyzing other robust phase retrieval algorithms.
\end{enumerate}

Finally, we remark that if we set the number of blocks $N=1$ in RCS, i.e., the algorithm reduces to the subgradient method,  the above convergence results yield new results for the subgradient method under the  linearly bounded subgradients assumption.

\subsection{Related works}\label{sec:related works}
\emph{The coordinate-type subgradient methods.} To the best of our knowledge, the exploration of extending the subgradient oracle into the coordinate regime remains relatively underexplored due to the set-valued nature of the subdifferential. In \cite{nesterov2014subgradient}, Nesterov proposed a randomized coordinate extension of the subgradient method for solving convex linear composite optimization problems. The convex linear composition structure is a special instance of our problem \eqref{eq:opt problem}.  The reported convergence rate result was achieved using the impractical Polyak's step size rule. The incorporation of randomization is crucial for connecting the analysis of their algorithm with the traditional subgradient method.  The work \cite{dang2015stochastic} follows the same randomization idea and establishes a similar convergence rate result for their coordinate-type subgradient method applied to general nonsmooth convex optimization problems. However, their algorithm needs to calculate the full subgradient at each iteration, and then randomly select one block of the calculated subgradient for the coordinate update, in order to avoid the intrinsic obstacle induced by the set-valued nature of the subdifferential; see the last paragraph of \cite[Section 1]{dang2015stochastic}. This approach negates the inherent advantage of coordinate-type methods over their full counterparts.  In theory, the results in \cite{nesterov2014subgradient,dang2015stochastic} only apply to convex and Lipschitz continuous optimization problems. Unfortunately, this limitation excludes  our three examples outlined in \Cref{sec:motivation and problem}. By contrast, 1) our proposed RCS applies to the more general composite problem class \eqref{eq:opt problem} and 2) our theoretical findings extend to non-Lipschitz and weakly convex cases.

\emph{Analysis under generalized Lipschitz-type assumptions.} The linearly bounded subgradients assumption dates back to \cite{CohenZ}, in which Cohen and Zhu showed the global convergence of the subgradient method for nonsmooth convex optimization. This assumption was later imposed in \cite{culioli1990} to analyze stochastic subgradient method for stochastic convex optimization. However, both works concern the global convergence property while no explicit rate result is reported.

Apart from the assumption we used, there are several other directions for generalizing the theory of  subgradient-type methods to non-Lipschitz convex optimization. In \cite{lu2019}, the author proposed the concept of relative continuity of the objective function, which is  to impose some relaxed bound on the subgradients by using Bregman distance. Then, the author invoked this Bregman distance in the algorithmic design. The resultant mirror descent-type algorithm has a different structure from the classic subgradient method.  Similar idea was later used in \cite{zhou2020} for analyzing online convex optimization problem. The works \cite{renegar2016,grimmer2018} transform the original non-Lipschitz objective function to a radial reformulation, which introduces a new decision variable.  Then, their algorithms are based on alternatively updating the original variable (by a subgradient step) and the new variable.  In \cite{grimmer2019},  a growth condition rather than Lipschitz continuity of $f$ was utilized to analyze the normalized subgradient method. The author obtained asymptotic convergence rate result in terms of the minimum of the suboptimality gap through a rather quick argument that originates in Shor's analysis \cite[Theorem 2.1]{shor2012minimization}. This convergence result finally applies to any convex function that only needs to be locally Lipschitz continuous.

The common feature of the above mentioned works lies in that they all consider convex optimization problems.   By contrast, our work studies both convex and nonconvex cases.

\subsection{Notation}
The notation in this paper is mostly standard.  Throughout this work, we use $\setX^*$ to denote the set of optimal solutions of~\eqref{eq:opt problem}. We assume $\setX^*\neq\emptyset$ without loss of generality. For any $x^*\in\setX^*$, we denote $f^*=f(x^*)$. If problem \eqref{eq:opt problem} is nonconvex, the set of its critical points is denoted by $\overline \setX$.
Note that the indices $i(k)$, $k=0,1,2,\ldots$ generated in RCS are random variables. Thus, RCS generates a stochastic process $\{x^k\}_{k\geq 0}$. Throughout this work, let $(\Omega,\mathcal F,\{\mathcal F_k\}_{k\geq 0},\mathbb{P})$ be a filtered probability space and let us assume that the sequence of iterates $\{x^k\}_{k\geq 0}$ is adapted to the filtration $\{\mathcal F_k\}_{k\geq 0}$, i.e., each of the random vectors $x^k : \Omega \to \R^d$ is $\mathcal F_k$-measurable, which is automatically true once we set
$
\mathcal{F}_{k}=\sigma(i(0),i(1),\ldots,i(k))
$. We use $\mathbb{E}_{i(k)}$ and $\mathbb{E}_{\mathcal{F}_{k}}$ to denote the  expectations taken over the random variable $i(k)$ and the filtration $\mathcal F_k$, respectively. We $\widetilde \calO$ and $\tilde o$ to denote $\calO$ and $o$ with hidden log terms, respectively.

\section{Preliminaries and A Generalized Lipschitz-type Assumption}


\subsection{Convexity, weak convexity, and subdifferential}\label{sec:cvx-wcvx-subgrad}
For a convex function $\psi:\R^d\rightarrow\R$, a vector $v\in \R^d$ is called a {subgradient} of $\psi$ at point $x$ if the subgradient inequality
$
\psi(y)\geq\psi(x)+\langle v,y-x\rangle
$
holds for all $y\in\R^d$.
The set of all subgradients of $\psi$ at $x$ is called the subdifferential of $\psi$ at $x$, denoted by $\partial\psi(x)$.

A function $\varphi:\R^d\rightarrow\R$ is $\rho$-weakly convex if
$\psi = \varphi  + \frac{\rho}{2}\|\cdot\|^2$ is convex. For such a $\rho$-weakly convex function $\varphi$, its subdifferential is given by \cite[Proposition 4.6]{Vial83}
\[
\partial\varphi(x)=\partial \psi(x)-\rho x,
\]
where $\partial \psi$ is the convex subdifferential defined above.
Additionally, it is well known that the $\rho$-weak convexity of $\varphi$ is equivalent to \cite[Proposition 4.8]{Vial83}
\begin{equation}\label{eq:wcvx inequality}
\varphi(y)\geq\varphi(x)+\langle\nu,y-x\rangle-\frac{\rho}{2}\|x-y\|^2
\end{equation}
 for all $x,y\in \R^d$ and $\nu \in \partial \varphi (x)$.

In the following lemma, we establish that the composite case \ref{C2} gives rise to the weak convexity of $f$.

\begin{lemma}\label{lemma:wcvx_composite} Suppose that the function $h:\R^n\rightarrow\R$ is $\rho_h$-weakly convex and $L_h$-Lipschitz continuous and the mapping $\Phi:\R^d\rightarrow\R^n$ is smooth, $L_{\Phi}^0$-Lipschitz continuous, and its Jacobian $\nabla\Phi$ is $L_{\Phi}$-Lipschitz continuous. Then, the composite function $f(x)=h(\Phi(x))$ is $\left(\rho_h(L_{\Phi}^0)^2+L_hL_{\Phi}\right)$-weakly convex.
\end{lemma}
\begin{proof}
For any $\zeta\in\partial h(\Phi(x))$, by the $\rho_h$-weak convexity of $h$ and \eqref{eq:wcvx inequality}, we have
\be\label{eq:wcvx_composite_1}
\begin{aligned}
h(\Phi(y))-h(\Phi(x))&\geq\langle\zeta,\Phi(y)-\Phi(x)\rangle-\frac{\rho_h}{2}\|\Phi(x)-\Phi(y)\|^2\\
&\geq\langle\nabla \Phi(x)^{\top}\zeta,y-x\rangle-\frac{\rho_h}{2}\|\Phi(x)-\Phi(y)\|^2+\langle\zeta, \Phi(y)-\Phi(x)-\nabla \Phi(x)(y-x)\rangle\\
&\geq\langle\nabla \Phi(x)^{\top}\zeta,y-x\rangle-\frac{\rho_h}{2}\|\Phi(x)-\Phi(y)\|^2-\|\zeta\|\cdot\|\Phi(y)-\Phi(x)-\nabla \Phi(x)(y-x)\|.
\end{aligned}
\ee
Since the Jacobian $\nabla\Phi$ is $L_{\Phi}$-Lipschitz continuous, we have that $\|\Phi(y)-\Phi(x)-\nabla \Phi(x)(y-x)\|\leq\frac{L_{\Phi}}{2}\|x-y\|^2$. By the $L_{\Phi}^0$-Lipschitz of $\Phi$, we obtain that $\|\Phi(x)-\Phi(y)\|^2\leq(L_{\Phi}^0)^2\|x-y\|^2$. By the $L_h$-Lipschitz of $h$, we obtain that $\|\zeta\|\leq L_h$. Finally, invoking these bounds into \eqref{eq:wcvx_composite_1} gives
\[
\begin{aligned}
h(\Phi(y))\geq h(\Phi(x))+\langle\nabla \Phi(x)^{\top}\zeta,y-x\rangle-\frac{\rho_h(L_{\Phi}^0)^2+L_hL_{\Phi}}{2}\|x-y\|^2,
\end{aligned}
\]
which yields the desired result.
\end{proof}
\subsection{Moreau envelope of weakly convex function}\label{subsec:me_wcvx}
If $f$ is $\rho$-weakly convex, proper, and lower semicontinuous, then the Moreau envelope function $f_{\lambda}(x)$ and the proximal mapping $prox_{\lambda,f}(x)$ are defined as \cite{Rockafellar}:
\begin{eqnarray}
	&&f_{\lambda}(x):=\min_{y}\left\{f(y)+\frac{1}{2\lambda}\|y-x\|^2\right\},\label{eq:ME}\\
	&&prox_{\lambda,f}(x):=\arg\min_{y}\left\{f(y)+\frac{1}{2\lambda}\|y-x\|^2\right\}.\label{eq:PM}
\end{eqnarray}
As a standard requirement on the regularization parameter $\lambda$, we will always use $\lambda <\frac{1}{\rho}$ in the sequel to ensure that $\nabla f_{\lambda}$ of a $\rho$-weakly convex function $f$  is well defined.

We collect a series of important and known properties of the Moreau envelope  in the following proposition; see, e.g., \cite[Propositions 1-4]{zhudenglizhao2021}.
\begin{proposition}\label{prop:wcvxf_me} Suppose that $f$ is a $\rho$-weakly convex function.   The following assertions hold:
	\begin{enumerate}[label=\textup{\textrm{(\alph*)}},topsep=0pt,itemsep=0ex,partopsep=0ex]
		\item $prox_{\lambda,f}(x)$ is well defined, single-valued, and Lipschitz continuous.
		\item $f_{\lambda}(x)\leq f(x)-\frac{1-\lambda\rho}{2\lambda}\|x-prox_{\lambda,f}(x)\|^2$.
		\item $\lambda \dist(0,\partial f(prox_{\lambda,f}(x)))\leq\|x-prox_{\lambda,f}(x)\|\leq\frac{\lambda}{1-\lambda\rho}\dist(0,\partial f(x))$.
		\item $\nabla f_{\lambda}(x)=\frac{1}{\lambda}(x-prox_{\lambda,f}(x))$ and it is Lipschitz continuous.
		\item $x=prox_{\lambda,f}(x)$ if and only if $0\in\partial f(x)$.
	\end{enumerate}
\end{proposition}

 We have the following corollaries of \Cref{prop:wcvxf_me}, which show that the weakly convex function and its Moreau envelope mapping have the same set of optimal solutions and stationary points.
\begin{corollary}\label{cor:me}
	Suppose that $f$ is a $\rho$-weakly convex function. Let $\setX_{\lambda}^*$  denote the  set of minimizers of the problem $\min_{x\in\setR^d} f_{\lambda}(x)$.  Then, the following statements hold:
	\begin{enumerate}[label=\textup{\textrm{(\alph*)}},topsep=0pt,itemsep=0ex,partopsep=0ex]
		\item $f_{\lambda}(x)\geq f^*$ for all $ x\in\setR^d$.
		\item $f_{\lambda}(x_{\lambda}^*)=f^*$ for all $x_{\lambda}^*\in\setX_{\lambda}^*$. Consequently, we have $\setX_{\lambda}^*=\setX^*$.
	\end{enumerate}
\end{corollary}

\begin{proof}
We first show part (a). By the definition of $f_{\lambda}(x)$, we have
\[
f_{\lambda}(x)=\min_{y\in \R^d}\left\{f(y)+\frac{\|x-y\|^2}{2\lambda}\right\}=f(prox_{\lambda,f}(x))+\frac{\|x-prox_{\lambda,f}(x)\|^2}{2\lambda}\geq f(prox_{\lambda,f}(x))\geq f^*.
\]
	We now show part (b). By argument (b) of~\Cref{prop:wcvxf_me}, we have for all $ x^*\in\setX^*$ and  $x_{\lambda}^*\in\setX_{\lambda}^*$ that
	\[
	f^*=f(x^*)\geq f_{\lambda}(x^*)\geq f_{\lambda}(x_{\lambda}^*) \geq f^*,
	\]
	where the first inequality follows from argument (b) of~\Cref{prop:wcvxf_me} and the last inequality is due to part (a) of this corollary.
	The above inequality yields that $f^*= f(x^*) = f_{\lambda}(x_{\lambda}^*)$ for all $ x^*\in\setX^*$ and  $x_{\lambda}^*\in\setX_{\lambda}^*$. Thus, we have $\setX^*=\setX_{\lambda}^*$.
\end{proof}

\begin{corollary}\label{cor:mecritical} Suppose that $f$ is a $\rho$-weakly convex function.   Let  $\overline{\setX}_{\lambda}$ be the set of critical points of the problem  $\min_{x\in\setR^d} f_{\lambda}(x)$.  Then, we have  $\overline{\setX}=\overline{\setX}_{\lambda}$.
\end{corollary}

\begin{proof}
This corollary is a direct consequence of parts (c) and (d) of~\Cref{prop:wcvxf_me}.
\end{proof}

\subsection{The supermartingale convergence theorem}
Next, we introduce a well known convergence theorem in stochastic optimization society.
\begin{theorem}[Supermartingale convergence theorem \cite{RS}]\label{thm:RS}
	Let $\{\Lambda^k\}_{k\in\mathbb{N}}$, $\{\mu^k\}_{k\in\mathbb{N}}$, $\{\nu^k\}_{k\in\mathbb{N}}$, and $\{\eta^k\}_{k\in\mathbb{N}}$ be four positive sequences of real-valued random variables adapted to the filtration $\{\xi_{k}\}_{k\in\mathbb{N}}$. Suppose the following recursion holds true:
	\[
	\mathbb{E}_{\xi_{k}} \left[\Lambda^{k+1} \right] \leq(1+\mu^k)\Lambda^k+\nu^k-\eta^k,\quad\forall \ k\in\mathbb{N},
	\]
	where
	\[
	\sum_{k\in\mathbb{N}}\mu^k<\infty\quad\mbox{and}\quad\sum_{k\in\mathbb{N}}\nu^k<\infty \quad\mbox{almost surely}.
	\]
	Then, the sequence $\{\Lambda^k\}_{k\in\mathbb{N}}$ almost surely converges to a finite random variable\footnote{A random variable $X$ is finite if $\mathbb{P}\left(\{\omega\in\Omega: X(\omega)=\infty\}\right)=0$.} $\bar \Lambda$ and $\sum_{k\in\mathbb{N}}\eta^k<\infty$ almost surely.
\end{theorem}

\subsection{The linearly bounded subgradients assumption}
Let us present a more general condition than the traditional Lipschitz continuity assumption in this subsection. A fundamental result in variational analysis and nonsmooth optimization is the equivalence between Lipschitz continuity and bounded subgradients for a very general class of functions \cite[Theorem 9.13]{Rockafellar}.
Thus, in order to generalize the Lipschitz continuity or bounded subgradients assumption (they are equivalent as outlined),  we impose the following \emph{linearly bounded subgradients} assumption.
\begin{assumption}\label{assump1} There exist constants $L_1\geq 0$,  $L_2>0$ such that the function $f$ in \eqref{eq:opt problem} satisfies
\[
 \|r\| \leq L_1\|x\|+L_2, \quad \forall x \in \domain f, \  \ r\in\partial f(x).
\]
\end{assumption}
When $L_1 = 0$, \Cref{assump1} reduces to the traditional Lipschitz continuity assumption, i.e., $f$ is $L_1$-Lipschitz continuous \cite[Theorem 9.13]{Rockafellar}. Nonetheless, it is significantly more general with $L_1>0$.  For instance,  problem \eqref{eq:svm} satisfies our assumption trivially due to the linearly bounded term $L_1 \|x\|$.  Moreover, considerably many weakly convex problems satisfy our assumption.  By checking the subdifferential, it is easy to verify that  our assumption holds for problem \eqref{eq:pr} (see \Cref{cor:pr} for its subdifferential) and the robust low-rank matrix recovery problem \cite{li2020nonconvex}. Indeed, the composite form $h(\Phi(x))$ in \eqref{eq:opt problem} with  $h$ being $L_h$-Lipschitz continuous and the map $\Phi$ having $L_{\Phi}$-Lipschitz continuous Jacobian satisfies our assumption. Note that $\nabla \Phi(x)^{\top}\zeta \in \partial f(x)$ for any $\zeta\in \partial h(\Phi(x))$ according to \eqref{eq:subdifferential}. Then, it is easy to see that $\|\nabla \Phi(x)^{\top}\zeta\|\leq L_{h}\|\nabla \Phi(x)\|\leq L_{h}\|\nabla \Phi(x)-\nabla \Phi(0)\|+\|\nabla \Phi(0)\|\leq L_{h} L_{\Phi}\|x\|+\|\nabla \Phi(0)\|$. Since $\nabla \Phi(x)$ is continuous, there exists a $L_2>0$ such that $\|\nabla \Phi(0)\| \leq L_2$, which, together with defining $L_1 = L_{h} L_{\Phi}$, establishes the claim.

\section{Randomized Coordinate Subgradient Method}\label{sec:RCS}

In this section, we introduce a coordinate-type subgradient method for solving problem~\eqref{eq:opt problem}. The convexity / weak convexity of $f$ is currently not required in algorithm design. We first reveal a structured property of the subdifferential of the composite objective function $f$ in \eqref{eq:opt problem}. According to the composite structure, its subdifferential is given by \cite[Theorem 10.6]{Rockafellar}
\begin{equation}\label{eq:subdifferential}
    \boxed{\quad \partial f(x)=\nabla \Phi(x)^{\top}\partial h(\Phi(x)) \quad}
\end{equation}
where $\nabla \Phi(x) \in \R^{n\times d}$ is the Jacobian of the smooth map $\Phi$ at $x$ and $\partial h(\Phi(x))$ is the subdifferential of the outer function $h$ at $\Phi(x)$. For any subgradient $\zeta\in \partial h(\Phi(x))$, we have $\zeta\in \R^n$.

One important feature of $\partial f(x)$ is that it enables block coordinate computation as it fully separates the smooth part $\nabla \Phi(x)$ and the nonsmooth part $\partial h(\Phi(x))$. Let us decompose $\R^d$ into $N$ subspaces $\R^d=\bigotimes_{i=1}^N\R^{d_i}$ with $d=\sum_{i=1}^Nd_i$. Correspondingly, the decision variable $x\in \R^d$ is decomposed into $N$ block coordinates $x = (x_1,\ldots, x_i, \ldots, x_N)$ with $x_i\in \R^{d_i}$ and the Jacobian $\nabla \Phi(x)$ is decomposed into $N$ column-wise blocks
\[
    \nabla \Phi(x) = \begin{bmatrix}
        \nabla_{1} \Phi(x), \ldots, \nabla_{i} \Phi(x), \ldots, \nabla_N \Phi(x)
    \end{bmatrix}
\]
where $\nabla_{i} \Phi(x) \in \R^{n \times d_i}$ represents the derivative of $\Phi$ with respective to the block coordinate $x_i\in \R^{d_i}$.
To compute the $i$-th block coordinate subgradient of the objective function $f$, i.e., the subgradient of $f$ with respect to the block coordinate $x_i\in \R^{d_i}$, we can follow the two steps:
\begin{enumerate}[label=\textup{\textrm{(S.\arabic*)}},topsep=0pt,itemsep=0.2ex,partopsep=0ex, leftmargin = 1.2cm]
\item \label{S1} Compute a subgradient $\zeta \in \partial h(\Phi(x))$ of the outer function $h$ at $\Phi(x)$.
\item \label{S2} Compute the block coordinate Jacobian $\nabla_i \Phi(x) \in \R^{n\times d_i}$ and then compute the $i$-th block coordinate subgradient $r_i$ of $f$ by $r_i = \nabla_i \Phi(x)^\top \zeta \in \R^{d_i}$.
\end{enumerate}
In this way, aggregating all the $N$ block coordinate subgradients results in a complete subgradient of $f$, i.e.,
\begin{equation}\label{eq:coordinate subgrad to subgrad}
    \R^d \ni r := \begin{bmatrix}
        r_1 \\ \vdots \\ r_N
    \end{bmatrix} =
    \left[\begin{array}{c}\nabla_1 \Phi(x)^{\top}\zeta\\ \vdots\\ \nabla_N \Phi(x)^{\top}\zeta\end{array}\right] \in \partial f(x).
\end{equation}
The key factor of the above development lies in that one can pre-compute and then fix the subgradient $\zeta$ of the nonsmooth part, while the block coordinate separation is achieved by separating the Jacobian of the smooth part $\nabla \Phi(x)$, which is well defined.

Now, we are ready to design the Randomized Coordinate Subgradient method (RCS) for solving problem~\eqref{eq:opt problem}. At $k$-th iteration, RCS first follows step \ref{S1} to compute a subgradient $\zeta^k$ of the outer function $h$ at $\Phi(x^k)$. Second, it selects a block coordinate index $i(k)$ from $\{1,\ldots, N\}$ uniformly at random. This random choice of the block coordinate index is motivated by \cite{nesterov2014subgradient} and is used to connect the analysis of RCS with the complete subgradient (see \Cref{lemma:inequality_iter} below). RCS then computes the $i(k)$-th block coordinate subgradient $r^k_{i(k)}$ of $f$ by following step \ref{S2}. Finally, RCS only updates the $i(k)$-th block coordinate $x_{i(k)}$ with a proper step size. The pseudocode of RCS is depicted in \Cref{alg:RCS}.

\begin{figure}[t]
\label{alg:LSB}
\begin{algorithm}[H]
\caption{RCS: Randomized Coordinate Subgradient Method for Solving \eqref{eq:opt problem}}
{\bf Initialization:}  $x^0$ and $\alpha_0$;
\begin{algorithmic}[1]
\For{$k=0,1,\ldots$}
\State Compute a subgradient $\zeta^k\in\partial h(\Phi(x^k))$ of the outer function $h$ at $\Phi(x^k)$;
\State Sample a block coordinate index $i(k)$ from $\{1,...,N\}$ uniformly at random;
\State Compute a block coordinate subgradient $r_{i(k)}^k=\nabla_{i(k)}\Phi(x^k)^{\top}\zeta^k$;
\State Update the step size $\alpha_{k}$ according to a certain rule;
\State Update  $x_{i(k)}^{k+1} = x_{i(k)}^{k} - \alpha_{k}r_{i(k)}^k$, while keep $x_{j}^{k+1} = x_{j}^{k}$ for all other $j\neq i(k)$.
\EndFor
\end{algorithmic}
\label{alg:RCS}
\end{algorithm}
\end{figure}

\subsection{A preliminary recursion for RCS}\label{sec:recursion RCS}

The randomization step in RCS helps us connect the analysis of RCS (using a coordinate subgradient) to that based on the complete subgradient in expectation. The result is displayed in the following lemma, which establishes a preliminary recursion for RCS.

\begin{lemma}[basic recursion]\label{lemma:inequality_iter}
	Let $\{x^k\}_{k\in \N}$ be the sequence of iterates generated by RCS  for solving problem \eqref{eq:opt problem}. Then, for all $x\in \R^d$,  the following hold:
	\begin{enumerate}[label=\textup{\textrm{(\alph*)}},topsep=0pt,itemsep=0ex,partopsep=0ex]
\item $\|x^k-x^{k+1}\|\leq\alpha_k\|r^k\|$.
\item  If \Cref{assump1} holds, then we have
\[
\mathbb{E}_{i(k)} \left[\|x-x^{k+1}\|^2\right]\leq\|x-x^k\|^2+\frac{\alpha_k^2(L_1\|x^k\|+L_2)^2}{N}-\frac{2\alpha_k}{N}\langle r^k,x^{k}-x\rangle.
\]
\end{enumerate}
\end{lemma}

\begin{proof} For all $x\in\R^d$, we have
	\begin{equation}\label{eq:VI}
		\begin{aligned}
			\|x-x^{k+1}\|^2&=\sum_{j\neq i(k)}\|x_j-x_{j}^{k}\|^2+\|x_{i(k)}-x_{i(k)}^{k}+\alpha_k r_{i(k)}^k\|^2\\
			&=\|x-x^{k}\|^2+2\langle\alpha_k r_{i(k)}^k,x_{i(k)}-x_{i(k)}^{k}\rangle+(\alpha_k\|r_{i(k)}^k\|)^2.
		\end{aligned}
	\end{equation}
	We first show part (a). By taking $x=x^k$ in~\eqref{eq:VI}, we obtain
	\[
	\|x^k-x^{k+1}\|=\alpha_k\|r_{i(k)}^k\|\leq\alpha_k\|r^k\|.
	\]
	We now show part (b). By the design of RCS and \eqref{eq:coordinate subgrad to subgrad}, we have
   \[
     \mathbb{E}_{i(k)}\langle r_{i(k)}^k,(x^k-x)_{i(k)}\rangle=\frac{1}{N}\langle r^k,x^k-x\rangle
   \]
   and
   \[
   \mathbb{E}_{i(k)}\|r_{i(k)}^k\|^2=\frac{1}{N}\|r^k\|^2.
   \]
   Taking expectation with respect to $i(k)$ on both sides of~\eqref{eq:VI}, together with the above equations, gives
   \[
\mathbb{E}_{i(k)}\left[\|x-x^{k+1}\|^2\right]\leq\|x-x^{k}\|^2+\frac{\alpha_k^2}{N}\|r^k\|^2-\frac{2\alpha_k}{N}\langle r^k,x^{k}-x\rangle.
	\]
	Part (b) follows from invoking~\Cref{assump1} in the above inequality.
\end{proof}

\section{Analytical Results}\label{sec:analytical results}

In this section, we  introduce several analytical results that are important for our subsequent convergence analysis --- especially for the weakly convex case. These results  are new to our knowledge and are of independent interest.

\subsection{A convergence lemma}
The following lemma is a strict generalization of \cite[Lemma 4]{CohenZ} from scalar case to mapping case. It will be the key tool for establishing the asymptotic convergence results for RCS. Note that we do not restrict the mapping $\Theta$ to be the gradient of  $f$, though we use it in this way in our analysis.
\begin{lemma}[convergence lemma]\label{lemma:lemma4in84} Consider the sequences $\{y^k\}_{k\in\mathbb{N}}$ in $\R^d$ and $\{\mu_k\}_{k\in\mathbb{N}}$ in $\R_+$. Let $\Theta:\R^d\rightarrow \R^m$ be  $L_{\Theta}$-Lipschitz continuous over $\{y^k\}_{k\in\mathbb{N}}$.  Suppose further, $\exists p>0$, such that:
\begin{eqnarray}
&\exists \ M\in\R_+ \text{ such that } \|y^k-y^{k+1}\|\leq M\mu_k\label{eq:cd1}, \;\forall k\in\mathbb{N},\\
&\sum_{k\in\mathbb{N}}\mu_k=\infty,\label{eq:cd2}\\
&\exists\ \bar{\Theta}\in\R^m \text{ such that } \sum_{k\in\mathbb{N}}\mu_k\|\Theta(y^k)-\bar{\Theta}\|^p<\infty.\label{eq:cd3}
\end{eqnarray}
Then, we have $\lim_{k\rightarrow\infty}\|\Theta(y^k)-\bar{\Theta}\|=0$.
\end{lemma}
\begin{proof}
	 For arbitrary $\varepsilon>0$, let $\mathbb{N}_{\varepsilon} :=\{k\in\mathbb{N}:\|\Theta(y^k)-\bar{\Theta}\|\leq\varepsilon\}$. It immediately follows from \eqref{eq:cd2} and \eqref{eq:cd3} that $\liminf_{k\rightarrow\infty} \|\Theta(y^k)-\bar{\Theta}\| = 0 $, which implies that $\mathbb{N}_{\varepsilon}$ is an infinite set. Let $\bar{\mathbb{N}}_{\varepsilon}$ be the complementary set of $\mathbb{N}_{\varepsilon}$ in $\mathbb{N}$. Then, we have
	 \begingroup
	 \allowdisplaybreaks
	\[		\varepsilon^p\sum_{k\in\bar{\mathbb{N}}_{\varepsilon}}\mu_k\leq \sum_{k\in\bar{\mathbb{N}}_{\varepsilon}}\mu_k\|\Theta(y^k)-\bar{\Theta}\|^p\leq\sum_{k\in\mathbb{N}}\mu_k\|\Theta(y^k)-\bar{\Theta}\|^p <\infty,
	\]
	\endgroup
	where the first inequality is due to the mapping $x\mapsto x^p$ is non-decreasing for all $p>0$, while the last inequality is from \eqref{eq:cd3}.
	Hence, for arbitrary $\delta>0$, there exits an integer $\mathfrak{n}(\delta)$ such that
	\[
	\sum_{\substack{\ell\geq\mathfrak{n}(\delta),  \ \ell\in\bar{\mathbb{N}}_{\varepsilon}}}\mu_\ell\leq\delta.
	\]
	Next, we take an arbitrary $\gamma>0$ and set $\varepsilon=\frac{\gamma}{2}$, $\delta=\frac{\gamma}{2L_{\Theta}M}$. For all $k\geq\mathfrak{n}(\delta)$, if $k\in\mathbb{N}_{\varepsilon}$, then we have $\|\Theta(y^k)-\bar{\Theta}\|\leq\varepsilon<\gamma$, otherwise $k\in\bar{\mathbb{N}}_{\varepsilon}$. Let $m$ be the smallest element in the set $\{\ell\in\mathbb{N}_{\varepsilon}:\ell\geq k\}$. Clearly, $m$ is finite since $\mathbb{N}_{\varepsilon}$ is an infinite set. Without loss of generality, we can assume $m>k$.  Then, we have
	\[
	\begin{aligned}
		\|\Theta(y^k)-\bar{\Theta}\|& \leq\|\Theta(y^k)-\Theta(y^m)\|+\|\Theta(y^m)-\bar{\Theta}\| \\
		&\leq L_{\Theta}\sum_{\ell=k}^{m-1}\|y^{\ell+1}-y^\ell\|+\varepsilon \\
		&\leq L_{\Theta}M\sum_{\ell=k}^{m-1}\mu_\ell+\frac{\gamma}{2}\\
		&=  L_{\Theta}M\sum_{k \leq \ell <m, \ \ell\in \bar{\mathbb{N}}_{\varepsilon}}\mu_\ell+\frac{\gamma}{2}\\
		&\leq L_{\Theta}M\sum_{\ell\geq\mathfrak{n}(\delta), \  \ell\in\bar{\mathbb{N}}_{\varepsilon}}\mu_\ell+\frac{\gamma}{2}\\
		&\leq\gamma.
	\end{aligned}
	\]
	Here, the second inequality follows from the Lipschitz continuity of $\Theta$ over $\{y^k\}_{k\in\mathbb{N}}$, the third inequality is from \eqref{eq:cd1}, and the  equality is due to the definition of $m$. Since $\gamma>0$ is taken arbitrarily, we have shown $\|\Theta(y^k)-\bar{\Theta}\|\rightarrow 0$ as $k\rightarrow \infty $.
\end{proof}

\subsection{Error bounds}\label{sec:error bound}
In our later analysis for the weakly convex case, the more general linearly bounded subgradients assumption (i.e., \Cref{assump1}) introduces additional difficulties.  We need the notion of \emph{global metric subregularity} of $\nabla f_{\lambda}$ to tackle the difficulty. In this subsection, we first identify sufficient conditions on when this error bound holds true.
Then, we establish that  the  concrete robust phase retrieval problem \eqref{eq:pr} satisfies this error bound.

The following lemma connects two types  of global error bounds for weakly convex functions.

\begin{lemma}[relation between two global error bounds]\label{prop:eb}
	For a weakly convex function $f:\R^d\rightarrow \R$, the following statements are equivalent:
	\begin{enumerate}[label=\textup{\textrm{(\alph*)}},topsep=0pt,itemsep=0ex,partopsep=0ex]
		\item $\partial f$ satisfies the global metric subregularity, i.e., there exists a constant $\kappa_1>0$ such that
		\[
		\dist(x,\overline{\setX})\leq\kappa_1 \dist(0,\partial f(x)),\qquad\forall x\in\setR^d.
		\]
		\item $f$ satisfies the  global proximal error bound, i.e., there exists a constant $\kappa_2>0$ such that
		\[
		\dist(x,\overline{\setX})\leq\kappa_2\|x-prox_{\lambda,f}(x)\|,\qquad\forall x\in\setR^d.
		\]
	\end{enumerate}
\end{lemma}
The main difficulty is to show that (a) implies (b). We construct a convergent sequence such that $f$ does not satisfy the global proximal error bound at the limit point of this sequence and then, prove this lemma by contradiction; see the following detailed proof.
\begin{proof}
	The direction ``(a)$\Leftarrow$(b)" directly follows from  part (c)  of~\Cref{prop:wcvxf_me}.
	It remains to show the direction ``(a)$\Rightarrow$(b)".  Let us assume on the contrary that   $f$ does not satisfy the global proximal error bound. Then, there exist a point $\tilde x\in \R^d$ and  $\kappa_2^j>0$, $\kappa_2^j\rightarrow \infty$, $x_j\rightarrow \tilde x$ as $j\rightarrow \infty$  such that
	\be\label{eq:contradict 1}
	\dist(x_j,\overline{\setX})>\kappa_2^j\|x_j-prox_{\lambda,f}(x_j)\| \quad \text{as} \quad j\rightarrow \infty.
	\ee
	Let $\operatorname{cl}\overline \setX$ denote the closure of $\overline \setX$. Note that $	\dist(x,\overline{\setX}) = \dist(x,\operatorname{cl}\overline{\setX})$ \cite[Proposition 1D.4]{dontchev2009}.  Then, by denoting $y\in \proj(prox_{\lambda,f}(x_j),\operatorname{cl}\overline{\setX})$,  we have
	\begin{align*}
		\dist(prox_{\lambda,f}(x_j),\overline{\setX})& = \|prox_{\lambda,f}(x_j)-y\|\\
        &= \|prox_{\lambda,f}(x_j) - x_j + x_j-y\|\\
		&\geq \|x_j-y\|  - \|x_j-prox_{\lambda,f}(x_j)\| \\
		&\geq \dist(x_j,\overline{\setX})  - \|x_j-prox_{\lambda,f}(x_j)\| \\
		&\geq (\kappa_2^j-1)\|x_j-prox_{\lambda,f}(x_j)\|\\
		&\geq(\kappa_2^j-1)\lambda \dist(0,\partial f(prox_{\lambda,f}(x_j))),
	\end{align*}
	where the second inequality is due to $\|x_j-y\| \geq \|x_j- \bar x_j\| =  \dist(x_j,\overline{\setX})$ with $\bar x_j\in \proj( x_j,\operatorname{cl}\overline{\setX})$, the third inequality is due to \eqref{eq:contradict 1}, and
	the  last inequality follows from part (c)  of~\Cref{prop:wcvxf_me}.
	Since $(\kappa_2^j-1)>\kappa_1$ as $j\to \infty$, we reach a contradiction to the global metric subregularity (a).
\end{proof}

Invoking $\|\nabla f_{\lambda}(x)\| = \frac{1}{\lambda} \|x-prox_{\lambda,f}(x)\|$ (see \Cref{prop:wcvxf_me} (d)) in \Cref{prop:eb} (b) and using the fact that $f$ and $f_{\lambda}$ has exactly the same set of critical points (see \Cref{cor:mecritical}),   we obtain the following result.

\begin{corollary}\label{cor:me_ms} Suppose that the subdifferential $\partial f$ of the weakly convex function $f$ satisfies the global metric subregularity property with parameter $\kappa_1$. Then, there exists a parameter $\kappa_2>0$ such that $\nabla f_{\lambda}$ satisfies the global metric subregularity property with parameter $\lambda \kappa_2$.
\end{corollary}

Thus, in order to establish the global metric subregularity property of $\nabla f_{\lambda}$, it suffices to show the global metric subregularity of the subdifferential $\partial f$. We give a sufficient condition for the latter in the following lemma.

\begin{lemma}[sufficient condition for global metric subregularity of $\partial f$, Theorem 3.2 of~\cite{Zheng2014}]\label{lem:suffic_ms}
Suppose that $\partial f$ is a piecewise linear multifunction and $\lim_{\dist(x,\overline{\setX})\rightarrow\infty}\dist(0,\partial f(x))=\infty$. Then, $\partial f$ satisfies the global metric subregularity.
\end{lemma}

As a concrete example, the sufficient condition in \Cref{lem:suffic_ms} holds for our running example, the robust phase retrieval problem \eqref{eq:pr} and hence, $\nabla f_{\lambda}$ of this problem satisfies the global metric subregularity due to \Cref{cor:me_ms}. We present the result in the following proposition.
\begin{proposition}[global metric subregularity of robust phase retrieval problem]\label{cor:pr}
	$\partial f$ of the robust phase retrieval problem \eqref{eq:pr} satisfies the global metric subregularity if $A$ has full column rank.
\end{proposition}
\begin{proof}
Recalling the sufficient condition for global metric subregularity of $\partial f$ in~\Cref{lem:suffic_ms}:
\begin{itemize}
\item[{\rm(i)}] $\partial f$ is a piecewise linear multifunction;
\item[{\rm(ii)}] $\lim_{dist(x,\bar{\mathbb{X}})\rightarrow\infty}dist(0, \partial f(x))=\infty$.
\end{itemize}
The (real-valued) robust phase retrieval problem can be formulated as
\[
\min_{x\in\mathbb{R}^d}\quad f(x)=\frac{1}{n}\|(Ax)^{\circ2}-b^{\circ2}\|_1=\frac{1}{n}\sum_{i=1}^n|\left(a_i^{\top}x\right)^{2}-b_i^2|,
\]
where $A=\left(a_1,...,a_n\right)^{\top}\in\mathbb{R}^{n\times d}$ is nonzero known sampling matrix, $b=(b_1,...,b_n)^{\top}\in\mathbb{R}^{n}$ is observed measurements and $\circ2$ is the Hadamard power. The subgradient of robust phase retrieval problem~\eqref{eq:pr} is
\[
\partial f(x) =\frac{2}{n}A^{\top}\left(Ax\circ\xi\right)=\frac{2}{n}\sum_{i=1}^na_i(a_i^{\top}x)\xi_i,
\]
with $\xi_i\in\left\{\begin{aligned}\{-1\}&\qquad(a_i^{\top}x)^2-b_i^2<0\\ [-1,1]&\qquad(a_i^{\top}x)^2-b_i^2=0\\ \{1\}&\qquad(a_i^{\top}x)^2-b_i^2>0\end{aligned}\right.$ and $\circ$ is the Hadamard product. It is easy to see the feasible region of the subdifferential function is divided into  finitely many polyhedrons, defined by $\{x: a_i^{\top}x< -b_i\}$, $\{x: a_i^{\top}x = -b_i\}$, $\{x: -b_i < a_i^{\top}x < b_i\}$, $\{x: a_i^{\top}x  = b_i\}$, $\{x: a_i^{\top}x > b_i\}$ for all $i=1,\ldots, n$.   On these regions, the subdifferential function is always linear multifunction. Therefore, robust phase retrieval problem satisfies (i) of the sufficient condition of metric subregularity.

Next we are going to show the real valued phase retrieval problem satisfies (ii) of the sufficient condition of metric subregularity. Intuitively, for any given $x\in\R^d$, we can divided the index of datasets $\{(a_1,b_1),...,(a_n,b_n)\}$ into three subsets:
\be\label{eq:pr_ms_0}
\begin{aligned}
\setI_1&=\{i\in\{1,...,n\}|(a_i^{\top}x)^2=b_i^2\},\\
\setI_2&=\{i\in\{1,...,n\}|(a_i^{\top}x)^2>b_i^2\},\\
\setI_3&=\{i\in\{1,...,n\}|(a_i^{\top}x)^2<b_i^2\}.
\end{aligned}
\ee
We can easily find the fact that $\setI_1\cap\setI_2=\emptyset$, $\setI_2\cap\setI_3=\emptyset$, $\setI_1\cap\setI_3=\emptyset$, and $\setI_1\cup\setI_2\cup\setI_3=\{1,...,n\}$. For given $x\in\R^d$, the subdifferential of robust phase retrieval problem can be rewritten as
\be\label{eq:pr_ms_1}
\begin{aligned}
\partial f(x)=&\frac{2}{n}\bigg{[}\sum_{i\in\setI_1} a_i\left(a_i^{\top}x\cdot[-1,1]\right)+\sum_{i\in\setI_2} a_ia_i^{\top}x
-\sum_{i\in\setI_3} a_ia_i^{\top}x\bigg{]}.
\end{aligned}
\ee
And
\be\label{eq:pr_ms_2}
\frac{2}{n}\left[\sum_{i\in\setI_1} a_i\left(a_i^{\top}x\cdot\delta_i\right)+\sum_{i\in\setI_2} a_ia_i^{\top}x-\sum_{i\in\setI_3} a_ia_i^{\top}x\right]\in\partial f(x),
\ee
with $\delta_i\in[-1,1]$, $i=1,...,n$. It follows that
\begingroup
\allowdisplaybreaks
\be\label{eq:pr_ms_3_1}
\frac{2}{n}\left\|\sum_{i\in\setI_1} a_i\left(a_i^{\top}x\cdot\delta_i\right)+\sum_{i\in\setI_2} a_ia_i^{\top}x-\sum_{i\in\setI_3} a_ia_i^{\top}x\right\|
=\frac{2}{n}\bigg{\|}\sum_{i=1}^na_ia_i^{\top}x + \sum_{i\in\setI_1}a_i\left(a_i^{\top}x\cdot\delta_i\right)-\sum_{i\in\setI_1}a_i\left(a_i^{\top}x\right)-2\sum_{i\in\setI_3} a_ia_i^{\top}x\bigg{\|}
\ee
By the fact that $A^{\top}A=\sum_{i=1}^na_ia_i^{\top}$,~\eqref{eq:pr_ms_3_1} follows that
\be\label{eq:pr_ms_3}
\begin{aligned}
	&\frac{2}{n}\left\|\sum_{i\in\setI_1} a_i\left(a_i^{\top}x\cdot\delta_i\right)+\sum_{i\in\setI_2} a_ia_i^{\top}x-\sum_{i\in\setI_3} a_ia_i^{\top}x\right\|\\
	\geq&\frac{2}{n}\bigg{[}\|A^{\top}Ax\|-\|-\sum_{i\in\setI_1}a_i\left(a_i^{\top}x\cdot\delta_i\right)+\sum_{i\in\setI_1}a_i\left(a_i^{\top}x\right)+2\sum_{i\in\setI_3} a_ia_i^{\top}x\|\bigg{]}\\
	\geq&\frac{2}{n}\left[\|A^{\top}Ax\|-2\sum_{i\in\setI_1}\|a_i\|\cdot|a_i^{\top}x|-2\sum_{i\in\setI_3}\|a_i\|\cdot|a_i^{\top}x|\right],
\end{aligned}
\ee
\endgroup
with $\delta_i\in[-1,1]$. Recalling~\eqref{eq:pr_ms_0}, we have $|a_i^{\top}\bar{x}|=|b_i|$, $\forall i\in\setI_1$ and $|a_i^{\top}\bar{x}|<|b_i|$, $\forall i\in\setI_3$. Then~\eqref{eq:pr_ms_3} yields that
\be\label{eq:pr_ms_4}
\frac{2}{n}\left\|\sum_{i\in\setI_1} a_i\left(a_i^{\top}x\cdot\delta_i\right)+\sum_{i\in\setI_2} a_ia_i^{\top}x-\sum_{i\in\setI_3} a_ia_i^{\top}x\right\|\geq\frac{2}{n}\left[\|A^{\top}Ax\|-2\sum_{i\in\setI_1}\|a_i\|\cdot|b_i|-2\sum_{i\in\setI_3}\|a_i\|\cdot|b_i|\right],
\ee
with $\delta_i\in[-1,1]$. By the full column rank of $A$, we denote the positive definite matrix $Q=A^{\top}A$, and $\sigma_{min}(Q^{\top}Q)$ be the minimum eigenvalue of matrix $Q^{\top}Q$. Then~\eqref{eq:pr_ms_4} yields that
\be\label{eq:pr_ms_5}
\frac{2}{n}\left\|\sum_{i\in\setI_1} a_i\left(a_i^{\top}x\cdot\delta_i\right)+\sum_{i\in\setI_2} a_ia_i^{\top}x-\sum_{i\in\setI_3} a_ia_i^{\top}x\right\|
\geq\frac{2}{n}\left[\sqrt{\sigma_{min}(Q^{\top}Q)}\|x\|-2\sum_{i\in\setI_1}\|a_i\|\cdot|b_i|-2\sum_{i\in\setI_3}\|a_i\|\cdot|b_i|\right],
\ee
with $\delta_i\in[-1,1]$. By the combination of~\eqref{eq:pr_ms_1} and~\eqref{eq:pr_ms_5}, we have that
\be\label{eq:pr_ms_6}
\dist(0,\partial f(x))
\geq\frac{2}{n}\left[\sqrt{\sigma_{min}(Q^{\top}Q)}\|x\|-2\sum_{i\in\setI_1}\|a_i\|\cdot|b_i|-2\sum_{i\in\setI_3}\|a_i\|\cdot|b_i|\right].
\ee
Again using the full column rank of $A$, we have that the critical point set $\overline{\setX}$ is bounded (See~\Cref{fact:pr_boundcritical}). Then by the combination of~\eqref{eq:pr_ms_6} and the boundness of $\overline{\setX}$, we obtain that $\dist(0,\partial f(x)) \rightarrow \infty$ as $\dist(x,\overline{\mathbb{X}}) \rightarrow \infty$.

Then we can conclude that under the assumption $A$ is full column rank, $\dist(0,\partial f(x)) \rightarrow \infty$ as $\dist(x,\overline{\mathbb{X}}) \rightarrow \infty$. Thus, the  sufficient condition (ii) is satisfied, which leads to the global metric subregularity of $\partial f(x)$ of the robust phase retrieval problem.
\end{proof}

\section{Convergence Analysis for Convex Case}\label{sec:cvx}
In this section, we present the convergence results of RCS for the convex case. We first present the $\widetilde \calO(1/\sqrt{k})$ convergence rate result \emph{in expectation} in the following theorem.
\begin{theorem}[convergence rate in expectation]\label{theo:rate_cvx} Suppose that $f$ in \eqref{eq:opt problem} is convex and \Cref{assump1} holds.  Moreover, suppose the step sizes  $\{\alpha_k\}_{k\in \N}$ satisfy
\be\label{eq:step size cond}
\sum_{k\in \N} \alpha_k=\infty \quad\text{and}\quad \sum_{k\in \N}\alpha_k^2\leq\overline{\mathfrak{a}}<\infty.
\ee
Let $\tilde{x}^k=\frac{\sum_{j=0}^k\alpha_j x^j}{\sum_{j=0}^k\alpha_j}$ for all $k\geq 0$. Then, $\forall k \in \N$, we have
\be\label{eq:rate_cvx}
\mathbb{E}_{\mathcal{F}_{k-1}}\left[ f(\tilde{x}^{k})-f^*\right] \leq \frac{N\dist^2(x^{0},\setX^*) + C_1 \sum_{j=0}^k\alpha_j ^2}{2\sum_{j=0}^{k}\alpha_j},
\ee
where $C_1>0$ is a constant defined in the proof.
Consequently, if the step sizes $\alpha_k = \frac{\Delta}{\sqrt{k+1}\log(k+2)}$ with some constant $\Delta >0$ for all $k\in \N$, then, we have for all $k\geq 0$
\be\label{eq:rate_cvx_sqrt}
\mathbb{E}_{\mathcal{F}_{k-1}}\left[ f(\tilde{x}^{k})-f^*\right] \leq \frac{ \log (k+2) \left( \frac{N\dist^2(x^0,\setX^*)}{2\Delta} +  \frac{C_1\Delta}{(\log 2)^2}\right) }{\sqrt{k+1}}.
\ee
\end{theorem}


Our \Cref{theo:rate_cvx} generalizes the existing results \cite{nesterov2014subgradient,dang2015stochastic} by assuming linearly bounded subgradients assumption (i.e., \Cref{assump1}) rather than the traditional Lipschitz continuity assumption. This generalization is vital to cover more general convex problems like SVM (see \eqref{eq:svm}), the robust regression problem with weight decay, etc; see \Cref{sec:introduction} about the generality of \Cref{assump1}. Technically speaking,  this more general assumption introduces the additional term  $\frac{4}{N}L_1^2\alpha_k^2 \|x^{k}-x^*\|^2$ in the recursion obtained in \Cref{lemma:inequality_iter_cvx} (see below), which introduces additional difficulties in order to derive convergence result. Our main idea in the proof of \Cref{theo:rate_cvx} is to show an upper bound on $\{\mathbb{E}_{\mathcal{F}_{k-1}}[\|x^{k}-x^*\|^2]\}_{k\in \N}$ (see \Cref{lemma:boundness_Exk} below), which allows us to transfer this additional term to $\calO(\alpha_{k}^2)$ and absorb it  into the error term. Then, we can establish a standard recursion (see \Cref{lemma:inequality_iter_cvx_dist}) and apply classic analysis to  derive the desired result.  Let us also mention that  the diminishing step sizes conditions in \eqref{eq:step size cond} is important for bounding the term $\{\mathbb{E}_{\mathcal{F}_{k-1}}[\|x^{k}-x^*\|^2]\}_{k\in \N}$. Next, we present the detailed proof of \Cref{theo:rate_cvx}.

\begin{proof}
We need the following three lemmas as preparations.
\begin{lemma}[preliminary recursion]\label{lemma:inequality_iter_cvx}
Suppose that $f$ in \eqref{eq:opt problem} is convex and \Cref{assump1} holds.  Then, for any  $x^*\in\setX^*$, we have
	\be\label{eq:lemma_recur_cvx_Lip}
	\mathbb{E}_{i(k)}\left[\|x^{k+1}-x^*\|^2\right] \leq \left(1+\frac{4}{N}L_1^2\alpha_k^2 \right) \|x^{k}-x^*\|^2+\alpha_k^2\frac{4L_1^2\|x^*\|^2+2L_2^2}{N} -\frac{2\alpha_k}{N}(f(x^k)-f^*).
	\ee
\end{lemma}

\begin{proof}
 The derivation is standard and can be found in \Cref{supp:inequality_iter_cvx}.
\end{proof}

\begin{lemma}[boundedness of $\{\mathbb{E}_{\mathcal{F}_{k-1}}{[\|x^{k}-x^*\|^2]}\}_{k\in \N}$]\label{lemma:boundness_Exk}
	Suppose that $f$ in \eqref{eq:opt problem} is convex and \Cref{assump1} holds.   Suppose further that the step sizes  $\{\alpha_k\}_{k\in \N}$ satisfy~\eqref{eq:step size cond}. Then, for any $x^*\in\setX^*$, we have
	\[
	\mathbb{E}_{\mathcal{F}_{k-1}}\left[\|x^{k}-x^*\|^2\right]\leq \sB_1 \quad\forall k\in\setN.
	\]
	with $\sB_1=\left[\|x^{0}-x^*\|^2+\frac{(4L_1^2\|x^*\|^2+2L_2^2)\overline{\mathfrak{a}}}{N}\right]e^{1+\frac{4}{N}L_1^2\overline{\mathfrak{a}}}>0$.
\end{lemma}

\begin{proof}
Upon taking expectation with respect to $\mathcal{F}_{k-1}$ on both sides of~\Cref{lemma:inequality_iter_cvx} and by the fact that $f(x^k)-f^*\geq0$, we obtain
\[
\begin{aligned}
\mathbb{E}_{\mathcal{F}_{k}}\left[\|x^{k+1}-x^*\|^2\right]
	\leq& \left(1+\frac{4}{N}L_1^2\alpha_k^2\right)\mathbb{E}_{\mathcal{F}_{k-1}}\left[\|x^{k}-x^*\|^2\right]+\alpha_k^2\frac{4L_1^2\|x^*\|^2+2L_2^2}{N}-\frac{2\alpha_k}{N}\mathbb{E}_{\mathcal{F}_{k-1}}\left[f(x^k)-f^*\right]\\
	\leq&\left(1+\frac{4}{N}L_1^2\alpha_k^2\right)\mathbb{E}_{\mathcal{F}_{k-1}}\left[\|x^{k}-x^*\|^2\right]+\alpha_k^2\frac{4L_1^2\|x^*\|^2+2L_2^2}{N}.
\end{aligned}
\]
Let $\beta^k=\mathbb{E}_{\mathcal{F}_{k-1}}\left[\|x^{k}-x^*\|^2\right]\prod_{j=0}^{k-1}\left(1+\frac{4}{N}L_1^2\alpha_j^2\right)^{-1}$ with $\beta^0=\|x^{0}-x^*\|^2$. Then, multiplying $\prod_{j=0}^{k}\left(1+\frac{4}{N}L_1^2\alpha_j^2\right)^{-1}$ on both side of the above inequality yields
\[
\begin{aligned}
	\beta^{k+1}& \leq\beta^k+\alpha_k^2\frac{4L_1^2\|x^*\|^2+2L_2^2}{N}\prod_{j=0}^{k}\left(1+\frac{4}{N}L_1^2\alpha_j^2\right)^{-1}\\
	&\leq\beta^k+\alpha_k^2\frac{4L_1^2\|x^*\|^2+2L_2^2}{N}.\qquad\qquad\mbox{(since $\prod_{j=0}^{k}\left(1+\frac{4}{N}L_1^2\alpha_j^2\right)^{-1}<1$)}.
\end{aligned}
\]
Unrolling this inequality for $0,1,\ldots, k$ yields
\[
\begin{aligned}
\beta^{k+1}&\leq\beta^0+\left(\sum_{j=0}^k\alpha_j^2\right)\frac{4L_1^2\|x^*\|^2+2L_2^2}{N}\leq\beta^0+\overline{\mathfrak{a}}\frac{4L_1^2\|x^*\|^2+2L_2^2}{N},
\end{aligned}
\]
which further implies
\be\label{eq:bound_expect}
\mathbb{E}_{\mathcal{F}_{k}}\left[\|x^{k+1}-x^*\|^2\right]
	\leq\left[\|x^{0}-x^*\|^2+\frac{\overline{\mathfrak{a}}(4L_1^2\|x^*\|^2+2L_2^2)}{N}\right]\prod_{j=0}^{k}\left(1+\frac{4}{N}L_1^2\alpha_j^2\right).
\ee

Next, we upper bound the term $\prod_{j=0}^{k}\left(1+\frac{4}{N}L_1^2\alpha_j^2\right)$ as
\[
\begin{aligned}
\prod_{j=0}^{k}\left(1+\frac{4}{N}L_1^2\alpha_j^2\right)=e^{\log\prod_{j=0}^{k}\left(1+\frac{4}{N}L_1^2\alpha_j^2\right)}=e^{\sum_{j=0}^{k}\log\left(1+\frac{4}{N}L_1^2\alpha_j^2\right)} \leq e^{1+\frac{4}{N}L_1^2\sum_{j=0}^{k}\alpha_j^2}\leq e^{1+\frac{4}{N}L_1^2\overline{\mathfrak{a}}}.\qquad\mbox{(by~\eqref{eq:step size cond})}
\end{aligned}
\]
Invoking this bound in \eqref{eq:bound_expect} gives
\[
\mathbb{E}_{\mathcal{F}_{k}}\|x^{k+1}-x^*\|^2 \leq\sB_1 \quad \forall k\geq 0,
\]
where $\sB_1=\left[\|x^{0}-x^*\|^2+\frac{(4L_1^2\|x^*\|^2+2L_2^2)\overline{\mathfrak{a}}}{N}\right]e^{1+\frac{4}{N}L_1^2\overline{\mathfrak{a}}}>0$.
\end{proof}

Based on the previous two lemmas, we can derive the following recursion on $\dist(x,\setX^*)$.
\begin{lemma}[key recursion for convex case]\label{lemma:inequality_iter_cvx_dist}
	Suppose that $f$ in \eqref{eq:opt problem} is convex and \Cref{assump1} holds.   Suppose further that the step sizes  $\{\alpha_k\}_{k\in \N}$ satisfy~\eqref{eq:step size cond}.  Then, for any $x^*\in\setX^*$, we have
	\be\label{eq:lemma_recur_cvx_Lip_dist}
	\mathbb{E}_{\mathcal{F}_{k}}\left[\dist^2(x^{k+1}, \setX^*)\right]\leq \mathbb{E}_{\mathcal{F}_{k-1}}\left[\dist^2(x^{k}, \setX^*)\right]+\alpha_k^2\frac{C_1}{N}-\frac{2\alpha_k}{N}\mathbb{E}_{\mathcal{F}_{k-1}}\left[f(x^k)-f^*\right],
	\ee
	where $C_1:=28L_1^2\sB_1+16L_1^2\|x^*\|^2+2L_2^2$.
\end{lemma}

\begin{proof}
Let $\operatorname{cl}\setX^*$ denote the closure of $\setX^*$.  Let $x_k^*={\proj}(x^k,\operatorname{cl}\setX^*)$. Taking expectation with respect to $\mathcal{F}_{k-1}$ on both sides of~\Cref{lemma:inequality_iter_cvx} with $x^*=x_k^*$ provides
	\be\label{eq:dist_recursion_1}
	\begin{aligned}
&\mathbb{E}_{\mathcal{F}_{k}}\left[\dist^2(x^{k+1},\setX^*)\right] -\mathbb{E}_{\mathcal{F}_{k-1}}\left[\dist^2(x^{k}, \setX^*)\right]\\
		\leq& -\frac{2\alpha_k}{N}\mathbb{E}_{\mathcal{F}_{k-1}}\left[f(x^k)-f^*\right]+ \bigg{(}\frac{4}{N}L_1^2\mathbb{E}_{\mathcal{F}_{k-1}}\left[\dist^2(x^{k}, \setX^*)\right]+\mathbb{E}_{\mathcal{F}_{k-1}}\left[\frac{4L_1^2\|x_k^*\|^2+2L_2^2}{N}\right]\bigg{)}\alpha_k^2,
	\end{aligned}
	\ee
	where we also used the fact $\dist(x^k,\setX^*)= \dist(x^k,\operatorname{cl}\setX^*) =\|x^{k}-x_k^*\|$ \cite[Proposition 1D.4]{dontchev2009} and $\dist(x^{k+1},\setX^*)\leq\|x^{k+1}-x_k^*\|$.
	Note that \Cref{lemma:boundness_Exk} ensures $\mathbb{E}_{\mathcal{F}_{k-1}}\left[\|x^k-x^*\|^2\right]\leq\sB_1$. Since $\|x^k\|^2\leq 2\|x^k-x^*\|^2+2\|x^*\|^2$, we have  $\mathbb{E}_{\mathcal{F}_{k-1}}\left[\|x^k\|^2\right]\leq2\sB_1+2\|x^*\|^2$. In addition, since $\mathbb{E}_{\mathcal{F}_{k-1}}\left[\dist^2(x^{k}, \setX^*)\right]\leq\mathbb{E}_{\mathcal{F}_{k-1}}\left[\|x^k-x^*\|^2\right]\leq\sB_1$ and $\|x_k^*\|^2\leq 2\dist^2(x^k,\setX^*)+2\|x^k\|^2$, we can obtain $\mathbb{E}_{\mathcal{F}_{k-1}}\left[\|x_k^*\|^2\right]\leq6\sB_1+4\|x^*\|^2$.  Invoking these bounds in \eqref{eq:dist_recursion_1} yields the desired result:
	\[
	\begin{aligned}
		\mathbb{E}_{\mathcal{F}_{k}}\left[\dist^2(x^{k+1},\setX^*)\right]\leq &\mathbb{E}_{\mathcal{F}_{k-1}}\left[\dist^2(x^{k}, \setX^*)\right]+\frac{C_1}{N}\alpha_k^2-\frac{2\alpha_k}{N}\mathbb{E}_{\mathcal{F}_{k-1}}\left[f(x^k)-f^*\right],
	\end{aligned}
	\]
	with $C_1= 28L_1^2\sB_1+16L_1^2\|x^*\|^2+2L_2^2$.
\end{proof}

With these three lemmas, we can provide the  proof of~\Cref{theo:rate_cvx}.  Unrolling the recursion in \Cref{lemma:inequality_iter_cvx_dist} and by the definition of $\tilde{x}^k$ and the convexity of  $f$, we have
	\[
		\mathbb{E}_{\mathcal{F}_{k-1}}\left[f(\tilde{x}^k)-f^*\right]
 \leq\mathbb{E}_{\mathcal{F}_{k-1}} \left[ \frac{\sum_{j=0}^k\alpha_j (f(x^j)-f^*)}{\sum_{j=0}^{k}\alpha_j} \right]\leq\frac{N\dist^2(x^{0},\setX^*)+ C_1\sum_{j=0}^k\alpha_j^2}{2\sum_{j=0}^{k}\alpha_j}, \quad \forall k\in \N,
	\]
	which shows \eqref{eq:rate_cvx}.
	To derive~\eqref{eq:rate_cvx_sqrt}, let $\alpha_k = \frac{\Delta}{\sqrt{k+1}\log(k+2)}$ for all $k\in \N$.   The first condition in~\eqref{eq:step size cond} is automatically satisfied.  By the integral comparison test, we have
	\[
	\sum_{j=0}^k\alpha_j^2 \leq  \alpha_0^2 + \int_{0}^{k} \frac{\Delta^2}{(t+1)\log^2(t+2)} dt \leq \frac{2\Delta^2}{(\log 2)^2}.
	\]
	Thus, the second condition in \eqref{eq:step size cond} is also satisfied.
	Plugging the above upper bound and the fact that $\sum_{j=0}^{k}\alpha_j \geq \sum_{j=0}^{k} \frac{\Delta}{\sqrt{k+1}\log(k+2)} = \frac{\Delta \sqrt{k+1}}{\log(k+2)}$ into~\eqref{eq:rate_cvx} yields the result.
	
	The proof is complete.
\end{proof}

We will  establish more convergence results for RCS in the remaining parts of this section, and all the following  results are \emph{not} reported in the existing works \cite{nesterov2014subgradient,dang2015stochastic}.

Next, we  consider the situation where $f$ further satisfies the global quadratic growth condition. Then, we can derive  a $\calO(1/k)$ convergence rate in expectation using diminishing step sizes.

\begin{theorem}[convergence rate in expectation under additional global quadratic growth]\label{cor:rate_qg}
Suppose that $f$ in \eqref{eq:opt problem} is convex and \Cref{assump1} holds.  Suppose further that $f$ satisfies the global quadratic growth condition with parameter $\kappa_3>0$, i.e., there exists a constant $\kappa_3>0$ such that
$\dist^2(x,\setX^*)\leq\kappa_3(f(x)-f^*),\forall x\in\setR^d$. Consider the step sizes $\alpha_k = {N\kappa_3}/({k+1})$ for all $k\geq 0$, which satisfies the conditions in \eqref{eq:step size cond}. Then, we have
\[
\mathbb{E}_{\mathcal{F}_{k}} [{\dist}^2\left({x}^{k+1},\setX^*\right)] \leq \frac{NC_1\kappa_3^2}{k+1} \quad \forall k \in \N.
\]
\end{theorem}

\begin{proof}
	The proof is relatively standard given the recursion shown in \Cref{lemma:inequality_iter_cvx_dist}; see \Cref{supp:rate_qg} for the detailed proof.
\end{proof}

The additional assumption in~\Cref{cor:rate_qg}, i.e., $f$ satisfies the global quadratic growth condition,  is not stringent for a set of $\ell_2$-regularized convex problems such as the robust regression problem with weight decay and the SVM problem. Since these problems are strongly convex (due to the $\ell_2$-regularizer),  the global quadratic growth condition is automatically satisfied. More generally, the work~\cite{necoara2019linear} discusses sufficient conditions for global quadratic growth, which is valid for a much broader class of functions than strongly convex ones.

Both \Cref{theo:rate_cvx} and \ref{cor:rate_qg} present the convergence rate in the sense of expectation, which holds for the average of infinitely many runs of the algorithm. In the following,  we will establish the \emph{almost sure} convergence results, which hold with probability 1 for each single run of RCS.

\begin{theorem}[almost sure asymptotic convergence]\label{theo:convergence_cvx}
Suppose that $f$ in \eqref{eq:opt problem} is convex and \Cref{assump1} holds. Consider the step sizes conditions \eqref{eq:step size cond}.  We have $\lim_{k\rightarrow\infty}f(x^k)=f^*$ almost surely, and $\lim_{k\rightarrow\infty}x^k = x^* \in \setX^*$ almost surely.
\end{theorem}

In order to deal with the more general~\Cref{assump1}, we conduct all the arguments in the intersection of the events  $\Omega_1:=\{w:\{x^k(w)_{k\in \N}\} \text{ is bounded}\}$ and the events $\Omega_2:=\{w: \sum_{k\in \N}\alpha_k \left(f\left(x^k(w)\right)-f^* \right) < \infty\}$.  Applying the supermartingale convergence theorem (see \Cref{thm:RS})) to the recursion in \Cref{lemma:inequality_iter_cvx} yields $\mathbb{P}(\Omega_1 \cap \Omega_2) = 1$. Eventually, the desired result in \Cref{theo:convergence_cvx} can be derived by applying our convergence lemma (i.e., \Cref{lemma:lemma4in84}) to any outcome $w \in \Omega_1 \cap \Omega_2$ since all the requirements in \Cref{lemma:lemma4in84} are satisfied for such an $w$. We provide the detailed proof in the following.

\begin{proof}
We restate the recursion in \Cref{lemma:inequality_iter_cvx} below:
\be\label{eq:dimRS}
\mathbb{E}_{i(k)}\left[\|x^{k+1}-x^*\|^2\right] \leq \left(1+\frac{4}{N}L_1^2\alpha_k^2 \right) \|x^{k}-x^*\|^2+\alpha_k^2\frac{4L_1^2\|x^*\|^2+2L_2^2}{N}
-\frac{2\alpha_k}{N}(f(x^k)-f^*).
\ee
Note that $\|x^{k}-x^*\|^2\geq0$, $\frac{2\alpha_k}{N}(f(x^{k})-f^*)\geq0$, and  $\sum_{k\in\N}\alpha_k^2<\infty$. Therefore, upon applying \Cref{thm:RS} to \eqref{eq:dimRS}, we obtain that $\lim_{k\rightarrow\infty}  \|x^{k}-x^*\|^2$ almost surely exists.  It immediately follows that the sequence $\{x^k\}_{k\in \N}$ is almost surely bounded, i.e., we have $\Omega_1 \subseteq \Omega$ such that
\be\label{eq:prob1}
\mathbb{P}\left(\left\{w\in\Omega_1: \{x^k(w)\}_{k\in \N} \text{ is bounded}\right\}\right)=1.
\ee
Applying \Cref{thm:RS} to \eqref{eq:dimRS} also provides  $\sum_{k\in \N}\alpha_k(f(x^{k})-f^*)<\infty$ almost surely. Thus, we have $\Omega_2 \subseteq \Omega$ such that
\be\label{eq:prob2}
\mathbb{P}\left(\left\{w\in\Omega_2: \sum_{k\in \N}\alpha_k \left(f\left(x^k(w)\right)-f^* \right) < \infty\right\}\right)=1.
\ee
Let us define $\Omega_3 = \Omega_1 \cap \Omega_2$.  By \eqref{eq:prob1}, statement (a) of~\Cref{lemma:inequality_iter} and~\Cref{assump1}, we have  $\|x^k(w)-x^{k+1}(w)\|\leq C_3(w)\alpha_k$ and $f$ is Lipschitz continuous over $\{x^k(w)\}_{k\in \N}$ for all $w\in\Omega_3$. Then, we can apply \Cref{lemma:lemma4in84} to obtain
\[
\lim_{k\rightarrow\infty}f\left(x^k(w)\right)=f^*,\quad\forall w\in\Omega_3.
\]
From \eqref{eq:prob1} and \eqref{eq:prob2}, it is clear that $\mathbb{P} \left( \Omega_3  \right)= 1$. Finally, we obtain that $f(x^k)\rightarrow f^*$ almost surely.

Next, we show that $x^k \to x^* \in \setX^*$ almost surely.  Let us consider an arbitrary outcome $w\in \Omega_3$. Since  $\{x^k(w)\}$  is bounded and $f\left(x^k(w)\right) \to f^*$ as shown above, we can extract a convergent subsequence $\{x^{k_i}(w)\}$ such that $x^{k_i}(w) \to x^* \in \setX^*$. This, together with the established fact that $\lim_{k\rightarrow\infty}  \|x^{k}-x^*\|^2$ almost surely exists, yields the desired result.
\end{proof}

\Cref{theo:convergence_cvx} asserts that RCS eventually finds an optimal solution of problem \eqref{eq:opt problem} for each single run almost surely.  It does not provide a rate result. Actually, we can further establish the almost sure asymptotic rate result.

\begin{corollary}[almost sure  asymptotic convergence rate]\label{cor:asymptoticrate_cvx}
Under the setting of \Cref{theo:convergence_cvx}. Consider the step sizes  $\alpha_k = \frac{\Delta}{\sqrt{k+1}\log(k+2)}$ with $\Delta>0$ and let $\tilde{x}^k=  \frac{1}{k + 1} \sum_{j=0}^{k}x^j$ for all $k\geq 0$, then we have
\[
f(\tilde x^k)-f^* \leq o\left(\frac{\log(k+2)}{\sqrt{k+1}}\right) \quad \text{as} \ k\rightarrow \infty  \quad \text{almost surely}.
\]
\end{corollary}

\begin{proof}
	Form the proof of \Cref{theo:convergence_cvx}, we have  $\sum_{k\in \N}\alpha_k(f(x^{k})-f^*)<\infty$ almost surely. Applying Kronecker's lemma with the fact that $\{1/\alpha_k\}$ being increasing, we have
	\[
	      \lim_{k\to \infty} \alpha_k \sum_{j=0}^{k} (f(x^{j})-f^*) = 0.
	\]
	  Then, the desired result follows by recognizing
	\[
	   (k+1)  \alpha_k (f(\tilde x^k) - f^*) \leq  \alpha_k \sum_{j =0}^k(f(x^{j})-f^*),
	\]
where we have used convexity of $f$ and definition of $\tilde{x}^k$.
\end{proof}


\section{Convergence Analysis for Weakly Convex Case} \label{sec:weakly_cvx}
In this section,  we turn to establish convergence results for RCS when $f$ in \eqref{eq:opt problem} is weakly convex.
The iteration complexity in expectation for RCS is established in the following theorem.
\begin{theorem}[iteration complexity in expectation]\label{theo:rate_wcvx} Suppose that $f$ in \eqref{eq:opt problem} is $\rho$-weakly convex and \Cref{assump1} holds. Let $f_{\lambda}$ be the Moreau envelope of $f$ with $\lambda<{1}/{\rho}$. Let the step size $\alpha_k = \frac{\Delta}{\sqrt{T+1}}, \forall k\geq 0$ for some constant $\Delta>0$, where $T$ is the pre-determined total number of iterations. The following assertions hold:
\begin{enumerate}[label=\textup{\textrm{(\alph*)}},topsep=0pt,itemsep=0ex,partopsep=0ex]
\item If further the parameter $L_1 = 0$ in \Cref{assump1}, then we have
\[
\min_{0\leq k\leq T} \mathbb{E}_{\mathcal{F}_{T-1}}\left[\|\nabla f_{\lambda}(x^k)\|^2\right]\leq\frac{\frac{2N(f_{\lambda}(x^0)-f^*)}{\Delta}+\frac{L_2^2\Delta}{\lambda}}{(1-\lambda\rho)\sqrt{T+1}}.
\]
\item Otherwise, if further $\partial f$ satisfies the global metric subregularity property with parameter $\kappa_1>0$ (see \Cref{prop:eb} (a) for definition), the set of critical points $\overline{ {\setX} }$ is bounded, and $\alpha_k\leq\frac{1-\lambda\rho}{8L_1^2\kappa_2^2\lambda}$, then  we have
\[
\min_{0\leq k\leq T} \mathbb{E}_{\mathcal{F}_{T-1}}\left[\|\nabla f_{\lambda}(x^k)\|^2\right]\leq\frac{\frac{4N(f_{\lambda}(x^0)-f^*)}{\Delta}+4C_2\Delta}{(1-\lambda\rho)\sqrt{T+1}},
\]
where $\kappa_2, C_2>0$ are constants defined in the proof.
\end{enumerate}
\end{theorem}

We note that, part (a) of \Cref{theo:rate_wcvx} presents the complexity result under the traditional Lipschitz continuity assumption on the objective function $f$. Part (b) of \Cref{theo:rate_wcvx} applies to the more general linearly bounded subgradients assumption (i.e., $L_1 >0$ in \Cref{assump1}) with an additional global metric subregularity assumption. For both cases, we obtain the same order of iteration complexity results.

Let us interpret this iteration complexity result in a conventional way. In order to achieve $\min_{0\leq k\leq T} \mathbb{E}_{\mathcal{F}_{T-1}}\left[\|\nabla f_{\lambda}(x^k)\|^2\right] \leq \varepsilon$, RCS needs at most $\calO(\varepsilon^{-4})$ number of iterations in expectation, which is in the same order as other subgradient-type methods for weakly convex minimization. This is reasonable since RCS with  $N=1$ recovers the subgradient method.


Technically speaking, the linearly bounded subgradients assumption (i.e., \Cref{assump1}) introduces (considerably) extra technical difficulties. The key idea in our proof is to use the global metric subregularity property of $\nabla f_\lambda$ to establish the approximate descent property of RCS (see \Cref{lemma:inequality_iter_wcvx} below).  Such an error bound condition is ensured by  \Cref{cor:me_ms} as $\partial f$ is assumed to be globally metric subregular.  Note that \Cref{cor:pr} shows that $\partial f$ of the concrete robust phase retrieval problem \eqref{eq:pr} satisfies the global metric subregularity property. Next, we provide the detailed proof of \Cref{theo:rate_wcvx}.

\begin{proof}
In order to prove \Cref{theo:rate_wcvx}, we first derive the recursion for weakly convex case as preparation.
\begin{lemma}[key recursion for weakly convex case]\label{lemma:inequality_iter_wcvx}
\begin{enumerate}[label=\textup{\textrm{(\alph*)}},topsep=0pt,itemsep=0ex,partopsep=0ex]
\item Under the setting of statement (a) of~\Cref{theo:rate_wcvx}. Then, we have
\[
	\mathbb{E}_{i(k)}\left[f_{\lambda}(x^{k+1})-f^*\right]\leq f_{\lambda}(x^k)-f^*+\alpha_k^2\frac{L_2^2}{2N\lambda}-\frac{(1-\lambda\rho)\alpha_k}{2N}\|\nabla f_{\lambda}(x^k)\|^2,
	\]
\item Under the setting of statement (b) of~\Cref{theo:rate_wcvx}. Denote $\max_{\bar{x}\in\overline{\setX}}\|\bar{x}\|^2\leq\sB_2$ with some $\sB_2>0$. Then, we have
	\[
	\mathbb{E}_{i(k)}\left[f_{\lambda}(x^{k+1})-f^*\right]\leq f_{\lambda}(x^k)-f^*+\alpha_k^2\frac{C_2}{N}-\frac{(1-\lambda\rho)\alpha_k}{4N}\|\nabla f_{\lambda}(x^k)\|^2,
	\]
	where $C_2=\frac{2L_1^2\sB_2+L_2^2}{\lambda}$ is a positive constant.
\end{enumerate}
\end{lemma}
\begin{proof}
By the $\rho$-weak convexity of $f$ and the statement (b) of~\Cref{lemma:inequality_iter} (set $x=\hat{x}^k=prox_{\lambda,f}(x^k)$), we have
\be\label{eq:bound0_wcvx}
\mathbb{E}_{i(k)}\left[\|\hat{x}^k-x^{k+1}\|^2\right]
\leq\|\hat{x}^k-x^{k}\|^2+\frac{\alpha_k^2(L_1\|x^k\|+L_2)^2}{N}-\frac{2\alpha_k}{N}\left[f(x^k)-f(\hat{x}^k)-\frac{\rho}{2}\|\hat{x}^k-x^k\|^2\right].
\ee
The definitions of Moreau envelope~\eqref{eq:ME} and proximal mapping~\eqref{eq:PM} imply
\be\label{eq:bound1_wcvx}
\mathbb{E}_{i(k)}\left[f_{\lambda}(x^{k+1})\right]\leq f(\hat{x}^k)+\frac{1}{2\lambda}\mathbb{E}_{i(k)}\left[\|\hat{x}^k-x^{k+1}\|^2\right].
\ee
Combining \eqref{eq:bound0_wcvx} and \eqref{eq:bound1_wcvx} yields
\[
\begin{aligned}
\mathbb{E}_{i(k)}\left[f_{\lambda}(x^{k+1})\right]
	\leq&f(\hat{x}^k)+\frac{1}{2\lambda}\|\hat{x}^k-x^{k}\|^2+\frac{\alpha_k^2(L_1\|x^k\|+L_2)^2}{2N\lambda}-\frac{\alpha_k}{\lambda N}\left[f(x^k)-f(\hat{x}^k)-\frac{\rho}{2}\|\hat{x}^k-x^k\|^2\right]\\
	=&f_{\lambda}(x^k)+\frac{\alpha_k^2(L_1\|x^k\|+L_2)^2}{2N\lambda}-\frac{\alpha_k}{\lambda N}\left[f(x^k)-f(\hat{x}^k)-\frac{\rho}{2}\|\hat{x}^k-x^k\|^2\right]\\
	=&f_{\lambda}(x^k)+\frac{\alpha_k^2(L_1\|x^k\|+L_2)^2}{2N\lambda}-\frac{\alpha_k}{\lambda N}[f(x^k)-f_{\lambda}(x^k)] -\frac{\alpha_k}{2\lambda N}(\frac{1}{\lambda}-\rho)\|\hat{x}^k-x^k\|^2.
\end{aligned}
\]
This, together with parts (b) and (d) of \Cref{prop:wcvxf_me},  gives
\be\label{eq:bound2_wcvx}
\begin{aligned}
	\mathbb{E}_{i(k)}[f_{\lambda}(x^{k+1})]\leq& f_{\lambda}(x^k)+\frac{\alpha_k^2(L_1\|x^k\|+L_2)^2}{2N\lambda}-\frac{(1-\lambda\rho)\alpha_k}{2N}\|\nabla f_{\lambda}(x^k)\|^2.
\end{aligned}	
\ee
Part (a) follows the above inequality by setting $L_1=0$.

We now turn to establish part (b). Let $\bar{x}^k\in \proj(x^k,\operatorname{cl}\overline{\setX})$, where $\operatorname{cl} \overline\setX$ denotes the closure of $\overline \setX$ as used in the proof of \Cref{prop:eb}.  It then follows from \eqref{eq:bound2_wcvx} that
\be\label{eq:bound3_wcvx}
\begin{aligned}
\mathbb{E}_{i(k)}[f_{\lambda}(x^{k+1}) - f^*]
	\leq& (f_{\lambda}(x^k)- f^*)+\frac{\alpha_k^2(L_1\|x^k\|+L_2)^2}{2N\lambda}-\frac{(1-\lambda\rho)\alpha_k}{2N}\|\nabla f_{\lambda}(x^k)\|^2\\
	\leq&(f_{\lambda}(x^k)-f^*)+\frac{\alpha_k^2(4L_1^2\|x^k-\bar{x}^k\|^2+4L_1^2\sB_2+2L_2^2)}{2N\lambda}-\frac{(1-\lambda\rho)\alpha_k}{2N}\|\nabla f_{\lambda}(x^k)\|^2\\
	\leq&(f_{\lambda}(x^k)-f^*)+\frac{4L_1^2\alpha_k^2\dist^2(x^k,\overline \setX)}{2N\lambda}+\frac{C_2}{N}\alpha_k^2-\frac{(1-\lambda\rho)\alpha_k}{2N}\|\nabla f_{\lambda}(x^k)\|^2,
\end{aligned}
\ee
where $C_2:=\frac{2L_1^2\sB_2+L_2^2}{\lambda}$. Recall  the assumption that $\partial f$ satisfies the global metric subregularity with $\kappa_1>0$.  It follows from \Cref{cor:me_ms}  that $\nabla f_{\lambda}(x)$  also satisfies the global metric subregularity with some parameter $\kappa_2\lambda>0$.  Note that $\overline{\setX}=\overline{\setX}_\lambda$ (see \Cref{cor:mecritical}). Hence, we have
\be\label{eq:use eb}
\dist(x^k,\overline \setX)\leq\kappa_2\lambda\|\nabla f_{\lambda}(x^k)\|.
\ee
Then, invoking \eqref{eq:use eb} with  $\alpha_k\leq\frac{1-\lambda\rho}{8L_1^2\kappa_2^2\lambda}$ in \eqref{eq:bound3_wcvx} provides the desired result.
\end{proof}

With this lemma, we can provide proof of~\Cref{theo:rate_wcvx}. Taking total expectation to the recursion in \Cref{lemma:inequality_iter_wcvx}, for part (a) we obtain
\[
\frac{(1-\lambda\rho)\alpha_k}{2N}\mathbb{E}_{\mathcal{F}_{k-1}}\big{[}\|\nabla f_{\lambda}(x^k)\|^2\big{]}\leq \ \mathbb{E}_{\mathcal{F}_k}\big{[} (f_{\lambda}(x^k)-f^*)-(f_{\lambda}(x^{k+1})-f^*)\big{]}+\frac{L_2^2}{2N\lambda}\alpha_k^2.
\]
For part (b), we have
\[
\frac{(1-\lambda\rho)\alpha_k}{4N}\mathbb{E}_{\mathcal{F}_{k-1}}\big{[}\|\nabla f_{\lambda}(x^k)\|^2\big{]}
\leq \ \mathbb{E}_{\mathcal{F}_k}\big{[} (f_{\lambda}(x^k)-f^*)-(f_{\lambda}(x^{k+1})-f^*)\big{]}+\frac{C_2}{N}\alpha_k^2.
\]
Summing the above two inequalities over $0,1,\ldots, T$ and invoking step sizes $\alpha_k = \frac{\Delta}{\sqrt{T+1}}$ gives the desired complexity results.
\end{proof}

Note that we also require the set of critical points $\overline{\setX}$ is bounded. This is not a strong condition for many weakly convex optimization problems. For instance,  we show that the set of critical points $\overline \setX$ of the robust phase retrieval problem \eqref{eq:pr} is bounded under a mild condition that  the data matrix $A$ has full column rank (same condition as that in \Cref{cor:pr}). This result is summarized in the following fact and its proof is deferred to \Cref{appen:pr_boundcritical}.
\begin{fact}[bounded $\overline \setX$ of robust phase retrieval]\label{fact:pr_boundcritical}
For the robust phase retrieval problem \eqref{eq:pr},  if the data matrix $A$ has full column rank, then the set of critical points $\overline \setX$  is bounded.
\end{fact}

In the following theorem,  we give the almost sure  convergence result that holds true for each single run of the algorithm.

\begin{theorem}[almost sure asymptotic convergence]\label{theo:convergence_wcvx}
Suppose that $f$ in \eqref{eq:opt problem} is $\rho$-weakly convex and \Cref{assump1} holds. Let $f_{\lambda}$ be the Moreau envelope of $f$ with $\lambda<{1}/{\rho}$. Consider the step sizes $\{\alpha_k\}_{k\in \N}$ satisfying \eqref{eq:step size cond}. If one of the following conditions hold:
\begin{enumerate}[label=\textup{\textrm{(\alph*)}},topsep=0pt,itemsep=0ex,partopsep=0ex]
\item The parameter $L_1 = 0$ in \Cref{assump1};
\item$\partial f$ satisfies the global metric subregularity property with parameter $\kappa_1>0$ (see \Cref{prop:eb} (a) for definition) and the set of critical points $\overline{\setX}$ is bounded,
\end{enumerate}
\noindent then $\lim_{k\rightarrow\infty}\|\nabla f_{\lambda}(x^k)\|=0$ almost surely and hence, every accumulation point of $\{x^k\}_{k\in \N}$ is a critical point of problem \eqref{eq:opt problem} almost surely.
\end{theorem}

\begin{proof}
	We repeat the recursion in \Cref{lemma:inequality_iter_wcvx} below for convenience (Note that $\alpha_k\to 0$ due to \eqref{eq:step size cond}. The requirement on $\alpha_k$ in \Cref{lemma:inequality_iter_wcvx} must be satisfied for all large enough $k$):
\be\label{eq:dimRS_wcvx_2_1}
	\mathbb{E}_{i(k)}[f_{\lambda}(x^{k+1})-f^*]\leq f_{\lambda}(x^k)-f^*+\frac{L_2^2\alpha_k^2}{2N}-\frac{(1-\lambda\rho)\alpha_k}{2N}\|\nabla f_{\lambda}(x^k)\|^2.
\ee
for part (a) and
\be\label{eq:dimRS_wcvx_2}
\mathbb{E}_{i(k)}[f_{\lambda}(x^{k+1})-f^*]\leq f_{\lambda}(x^k)-f^*+\alpha_k^2\frac{C_2}{N}-\frac{(1-\lambda\rho)\alpha_k}{4N}\|\nabla f_{\lambda}(x^k)\|^2,
\ee
for part (b). Noted that $f_\lambda(x)-f^*\geq 0$, $\|\nabla f_{\lambda}(x^k)\|^2\geq0$, and $\sum_{k=0}^{\infty}\alpha_k^2<\infty$. Applying \Cref{thm:RS} to~\eqref{eq:dimRS_wcvx_2} (or~\eqref{eq:dimRS_wcvx_2_1}) yields $\lim_{k\rightarrow\infty}\left(f_\lambda(x^k)-f^*\right)$ almost surely exists and is finite. By the definition of $f_\lambda$, we have $\{f(prox(x^k)) + \frac{1}{2\lambda} \|x^k - prox(x^k)\|^2 -f^*\}_{k\in \N}$ is almost surely bounded, which further implies that $\{\|\nabla f_\lambda (x^k)\|\}_{k\in \N}$ is almost surely bounded due to $f(prox(x^k)) \geq f^*$ and $\|x^k - prox(x^k)\| = \lambda \|\nabla f_\lambda (x^k)\|$.

For part (a), $L_1=0$ in \Cref{assump1} yields the boundedness of $r^k$.

For part (b), note that $\partial f$ satisfies the global metric subregularity. We have $\dist(x^k,\overline \setX) \leq\kappa_2\lambda\|\nabla f_{\lambda}(x^k)\|$ for some $\kappa_2>0$ as derived in \eqref{eq:use eb}. Thus, $\{x^k\}_{k\in \N}$ is almost surely bounded since $\overline \setX$ is bounded by our assumption.  Next, combining~\Cref{assump1} and the almost surely boundedness of $\{x^k\}$ gives the boundedness of $r^k$.

By statement (a) of~\Cref{lemma:inequality_iter}, for both cases we have the events $\Omega_4$ such that
\[
\mathbb{P}\left(\left\{w\in\Omega_4:\begin{aligned}\|x^k(w)-x^{k+1}(w)\|\leq C_4(w)\alpha_k, C_4(w)>0\end{aligned}\right\}\right)=1.
\]

Applying \Cref{thm:RS} to~\eqref{eq:dimRS_wcvx_2_1} and~\eqref{eq:dimRS_wcvx_2} also provides $\Omega_5$ such that
\be\label{eq:as_rate_contradiction}
\mathbb{P}\left(\left\{w\in\Omega_5: \sum_{k=0}^{\infty}\alpha_k\|\nabla f_{\lambda}(x^k(w))\|^2<\infty\right\}\right)=1.
\ee
Noted that $\sum_{k=0}^{\infty}\alpha_k=\infty$ and $\nabla f_{\lambda}$ is Lipschitz continuous. By~\Cref{lemma:lemma4in84}, we have
\[
\lim_{k\rightarrow\infty}\|\nabla f_{\lambda}(x^k(w))\|=0, \quad \forall w\in \Omega_4 \cap \Omega_5.
\]
It is easy to see that $\mathbb{P}\left(\Omega_4\cap\Omega_5\right)=1$. Hence, we have
\[
\lim_{k\rightarrow\infty}\|\nabla f_{\lambda}(x^k)\|=0\quad\mbox{almost surely.}
\]
Invoking \Cref{cor:mecritical}, we conclude that every accumulation point of $\{x^k\}_{k\in \N}$ is a critical point of~\eqref{eq:opt problem} almost surely.
\end{proof}

Finally, we can derive the almost sure asymptotic convergence rate in the sense of liminf in the following corollary.
\begin{corollary}[almost sure convergence rate]\label{cor:asymptoticrate_wcvx}
Under the setting of either part (a) or part (b) of \Cref{theo:convergence_wcvx}. Let $\alpha_k=\frac{\Delta}{\sqrt{k+1}\log(k+2)}$ with some $\Delta>0$. We have $\liminf_{k\rightarrow\infty}(k+1)^{\frac{1}{4}}\|\nabla f_{\lambda}(x^k)\|=0$ almost surely.
\end{corollary}
\begin{proof}
  The derivation is based on \Cref{theo:convergence_wcvx}; see \Cref{supp:asymptoticrate_wcvx}.
\end{proof}
\section{Numerical Simulations} \label{sec:num}
In this section, we test RCS on Applications 1-3 listed in \Cref{sec:introduction}. From these experiments, our  observations can be summarized as follows: RCS uses much less workspace memory than the subgradient method (SubGrad) per iteration. Moreover, RCS is observed to converge faster than SubGrad in the first few epochs. Here, one epoch means passing through all the $d$ variables in $x$. For instance, one iteration of SubGrad means one epoch, while one epoch of RCS consists of $N$ iteration if we set $N = d$ in \Cref{alg:RCS}. These observations justify the main spirit of coordinate-type methods: RCS could be very efficient when the dimension of the problem is too high to use SubGrad (out of memory for even one iteration) and when the solution accuracy just needs to be modest (then just run RCS for a few epochs).  Note that these two situations are common in signal processing and machine learning areas.

The experiments on the robust M-estimators and linear SVM problems are conducted by using MATLAB R2020a on a personal computer with Intel Core i5-6200U CPU (2.4GHz) and 8 GB RAM. The experiments on the robust phase retrieval problem are conducted by using Python on a computer cluster with $2\times$ Intel Xeon Cascade Lake 6248 (2.5GHz, 20 cores) CPU and $12\times$ Samsung 16GB DDR4 ECC REG RAM.

\subsection{Robust M-estimators problem}\label{sec:M-esti}
We compare RCS with SubGrad on the robust M-estimators problem \eqref{eq:M-esti} with $\ell_1$-loss and $\ell_1$-penalty, i.e., Application 1 with $\ell(\cdot)=\frac{1}{n}\|\cdot\|_1$ and $\phi_p(\cdot)=p|\cdot|$:
\[
\min\limits_{x\in\R^d} \ f(x) := \frac{1}{n}\|Ax-b\|_1+p\|x\|_1,
\]
where $A\in\R^{n\times d}$ is a matrix,  $b\in\R^n$ is a vector, and $p>0$ is the regularization parameter.

In order to implement RCS, we introduce the block-wise partition of the columns of the data matrix $A=(A_{1},A_{2},\cdots,A_{N})$, where $A_{i} \in \R^{n\times d_{i}}$ is the $i$-th block.
The key factor is to calculate the coordinate subgradient $r^k_{i(k)}$ used in RCS for solving this problem, which is give by:
\[
\left\{
\begin{aligned}
	&r_{i(k)}^k =  \frac{1}{n}A_{i(k)}^{\top}\sign(s^k)+p\sign(x_{i(k)}^k),\\
	&s^{k+1}=s^k+A_{i(k)}(x_{i(k)}^{k+1}-x_{i(k)}^{k}),
\end{aligned}
\right.
\]
where $s^0=Ax^0-b$. $s^k$ is an iteratively updated intermediate quantity that is used to compute $r^k_{i(k)}$.

We generate synthetic data for simulation. The elements of $A\in\R^{n\times d}$ are generated in an i.i.d. manner according to Gaussian $\mathcal{N}(0,1)$ distribution. To construct a sparse true solution $x^*\in\R^d$, given the dimension $d$ and sparsity $s$, we select $s$ entries of $x^*$ uniformly at random  and fill the selected entries with values following i.i.d. $\mathcal{N}(0,1)$ distribution, and set the rest to zero.  To generate
the outliers vector $\delta\in\R^n$, we first randomly select $(p_{\mbox{fail}}\cdot n)$ locations. Then, we fill each of the selected locations with an i.i.d.  Gaussian $\mathcal{N}(0,1000)$ entry, while the remaining locations are set to $0$. Here, $p_{\mbox{fail}}$ is the ratio of outliers. The measurement vector $b$ is obtained by $b=Ax^*+\delta$.

In~\Cref{fig:00}, we show the evolution of $f(x^k)-f^*$ and $\dist(x^k-x^*)$ versus epoch counts with different number of blocks $N$.  We can observe that RCS with larger $N$ converges faster than RCS with smaller $N$ and SubGrad.   In addition, larger $N$ requires less workspace memory than small $N$ and SubGrad per iteration. Moreover, we find that RCS can effectively recover the sparse coefficients.
In \Cref{num:MCPMestimator}, we provide more experiments on the robust M-estimators problem with MCP-loss and $\ell_1$-penalty.

\begin{figure}[t]
\begin{minipage}[t]{0.24\linewidth}
\centering
\includegraphics[width=1.7in]{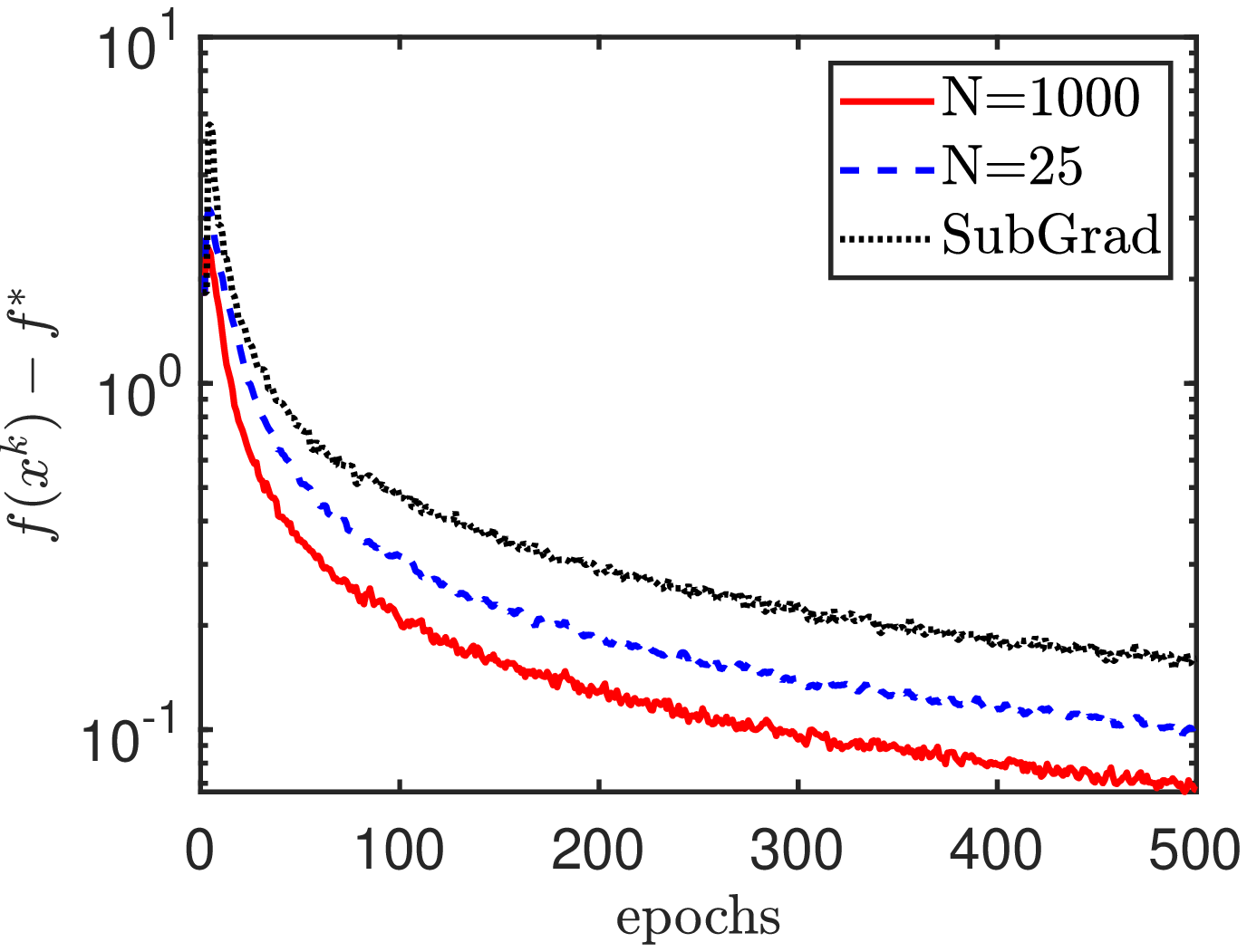}
\end{minipage}
\begin{minipage}[t]{0.24\linewidth}
\centering
\includegraphics[width=1.7in]{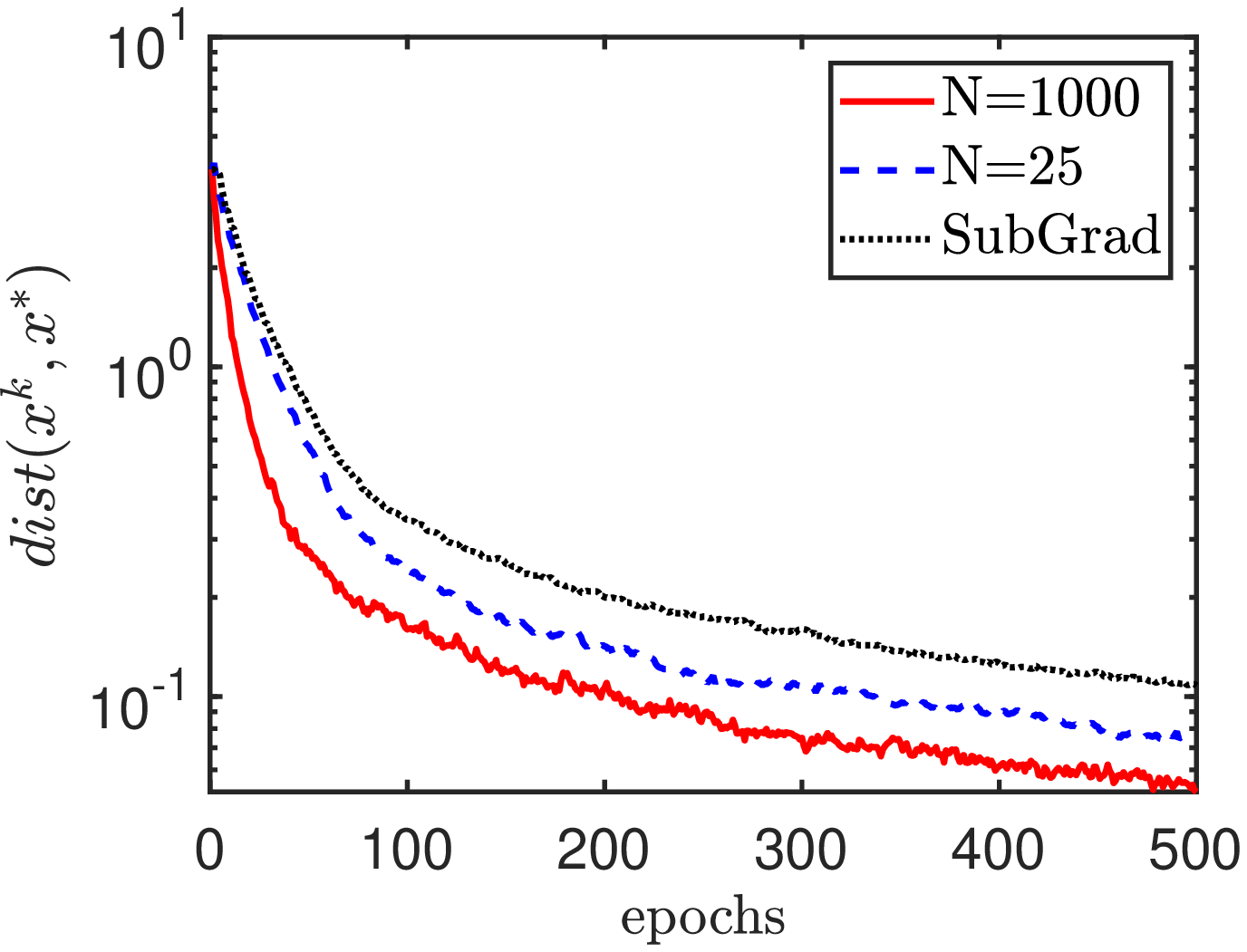}
\end{minipage}
\begin{minipage}[t]{0.24\linewidth}
\centering
\includegraphics[width=1.7in]{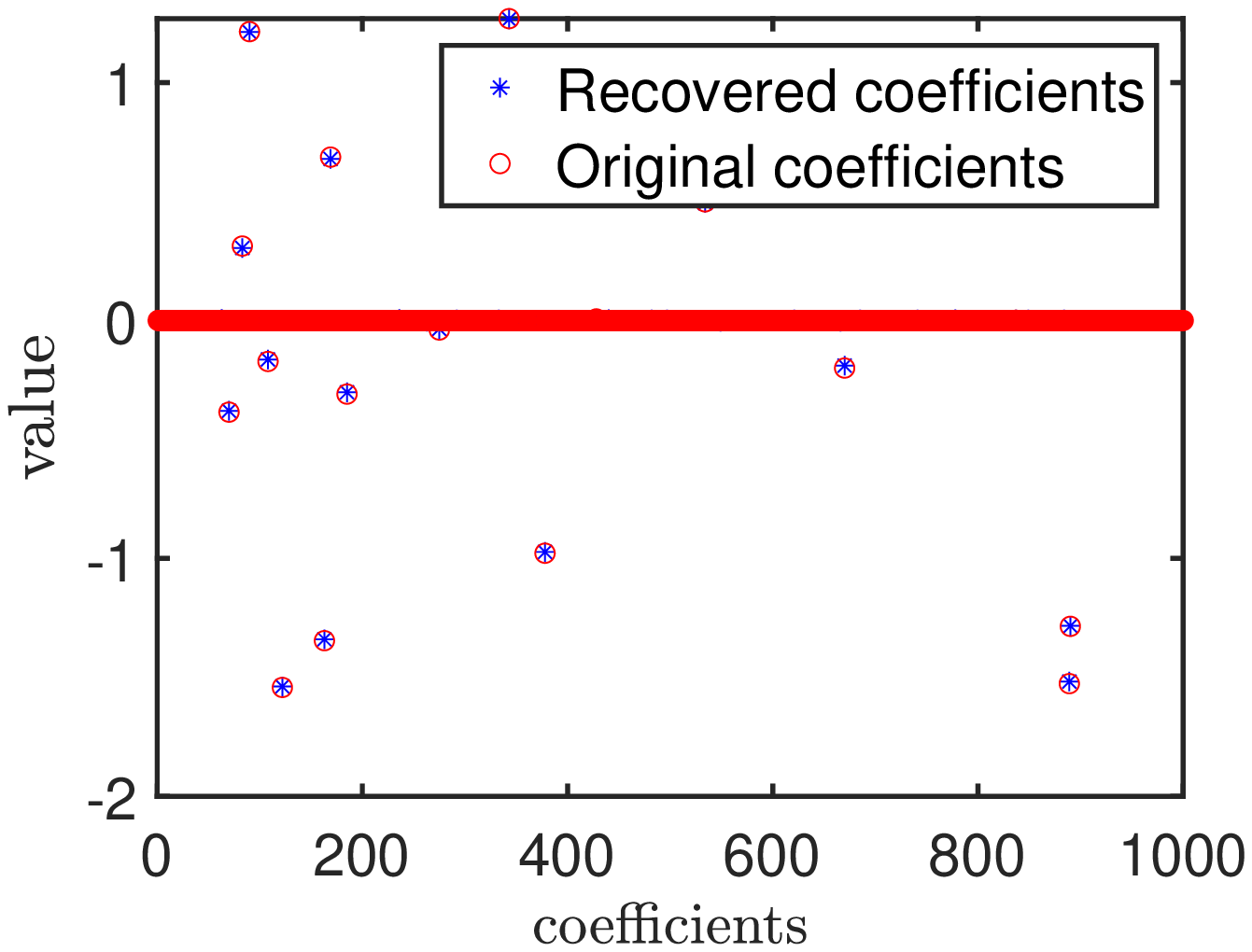}
\end{minipage}
\begin{minipage}[t]{0.24\linewidth}
\centering
\includegraphics[width=1.7in]{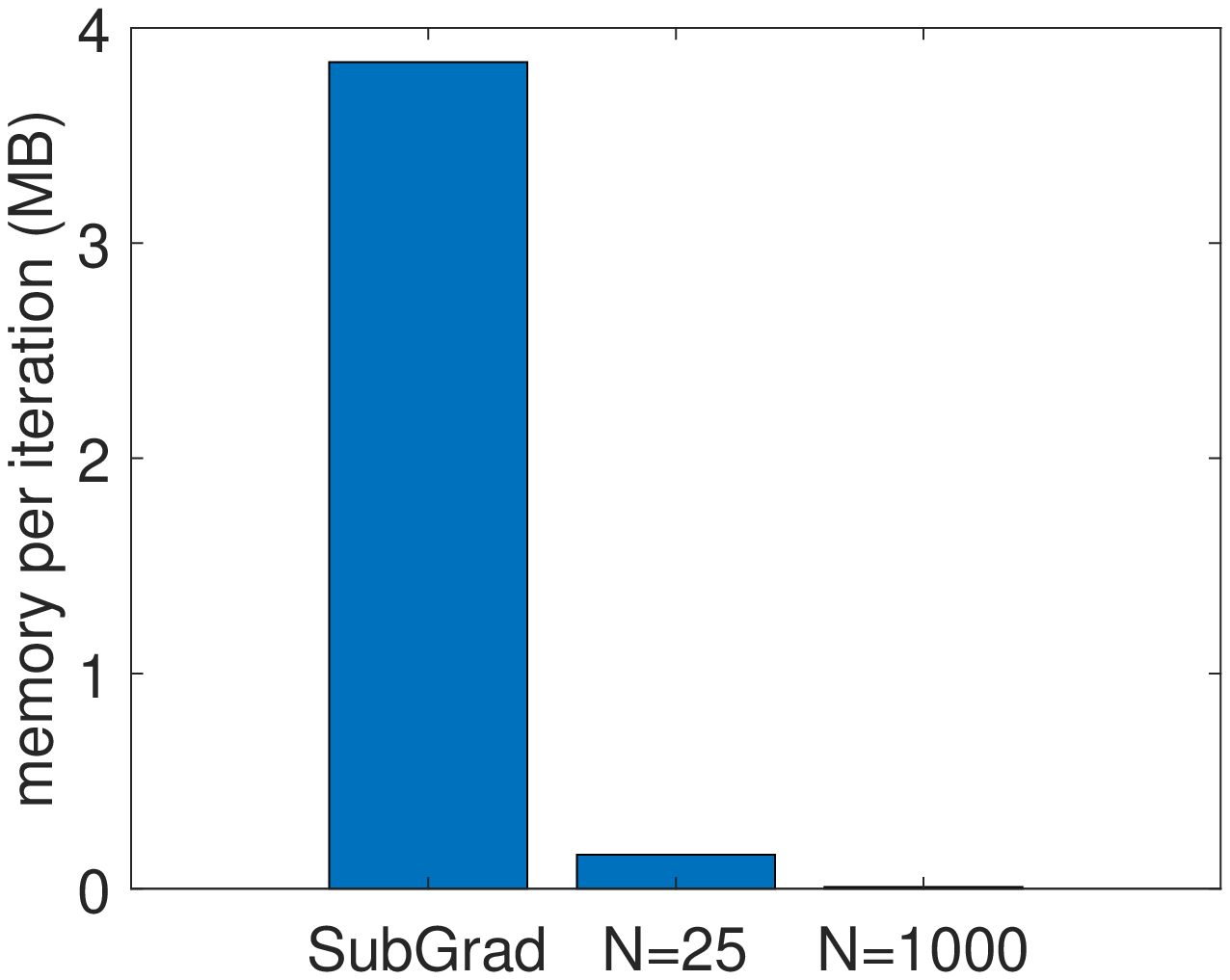}
\end{minipage}
\caption{\small Experiments on the robust M-estimators problem with $\ell_1$-loss and $\ell_1$-penalty.  {\tt Top-left:} $f(x^k)-f^*$ versus epoch counts for different choices of $N$ (i.e., the number of blocks in RCS); {\tt Top-right:} $\dist(x^k,x^*)$ versus epoch counts for different choices of $N$; {\tt Bottom-left:} Coefficients recovery by RCS with $N = d$; {\tt Bottom-right:} Comparison on workspace memory consumption per iteration. Here, $n=500$, $d = 1000$, $s = 20$, $p_{\mbox{fail}} = 0.2$.}\label{fig:00}
\end{figure}

\begin{table}[t]
	\caption{RCS versus SubGrad on linear SVM.}
	\label{table}
	\begin{center}
			\setlength{\tabcolsep}{1.4mm}{
					\begin{tabular}{lccccccccc}
						\toprule
						Data set                                                                &RCS              &                   &SubGrad          &                 \\
						($n/d$)                                                              & $f(x^k)$        &memory (MB)        &$f(x^k)$         &memory (MB)
\\
						\midrule
						\begin{tabular}{l}{\rm colon-cancer}\\($62/2000$)\end{tabular}          & 0.0379          & 0.0024            & 0.0399          & 1.9698          \\
						\begin{tabular}{l}{\rm duke}\\($44/7129$)\end{tabular}                  & 0.0078          & 0.0017            & 0.0719          & 5.0593          \\
						\begin{tabular}{l}{\rm leu}\\($38/7129$)\end{tabular}                   & 0.0062          & 0.0015            & 0.1325          & 4.4065          \\
						\begin{tabular}{l}{\rm gisette-scale.t}\\($1000/5000$)\end{tabular}     & 0.1270          & 0.0382            & 0.3166          & 76.5076         \\
						\begin{tabular}{l}{\rm a1a}\\($1605/119$)\end{tabular}                  & 0.5762          & 0.0613            & 0.5762          & 2.9557          \\
						\begin{tabular}{l}{\rm a2a}\\($2265/119$)\end{tabular}                  & 0.5896          & 0.0865            & 0.5896          & 4.1692          \\
						\begin{tabular}{l}{\rm a3a}\\($3185/122$)\end{tabular}                  & 0.5718          & 0.1216            & 0.5718          & 6.0067          \\
						\begin{tabular}{l}{\rm a4a}\\($4781/122$)\end{tabular}                  & 0.5821          & 0.1824            & 0.5821          & 9.0143          \\
						\bottomrule
					\end{tabular}
}
	\end{center}
\end{table}

\subsection{Linear SVM problem}
We compare RCS (set $N=d$)  with SubGrad on the linear SVM problem \eqref{eq:svm} (i.e., Application 2). We repeat the problem in the following for convenience:
\[
\min_{x\in\R^d}\ f(x) := \frac{1}{n}\sum_{i=1}^n\max\left\{0,1-b_i(a_i^{\top}x)\right\}+\frac{p}{2}\|x\|^2.
\]
where $a_i\in\R^{d}$ is a vector, $b_i\in\R$ is a scalar,  $p>0$ is a positive number.

Let $A=\left(a_1,...,a_n\right)^{\top}\in\R^{n\times d}$ be the data matrix and $b=\left(b_1,...,b_n\right)^{\top}\in\R^n$ be a vector.  We define $\R^{n\times d}\ni\tilde{A}=(b\mathbf{1}_d^{\top})\circ A$, where $\mathbf{1}_d$ is a $d$-dimensional vector of all $1$s and $\circ$ is the Hadamard product.
To implement RCS, we introduce the block-wise partition of the columns of the matrix $\tilde{A}=\left(\tilde{A}_{1},\tilde{A}_{2},\cdots,\tilde{A}_{N}\right)$, where $\tilde A_{i} \in \R^{n\times d_{i}}$ is the $i$-th block.
Then,  it is key to compute the coordinate subgradient $r^k_{i(k)}$, which can be calculated as:
\[
\left\{
\begin{aligned}
	&r_{i(k)}^k = -\frac{1}{n}\tilde{A}_{i(k)}^{\top}\max\{0,\sign(s^k)\}+px_{i(k)}^k,\\
	&s^{k+1}=s^k+\tilde{A}_{i(k)}(x_{i(k)}^{k}-x_{i(k)}^{k+1}),
\end{aligned}
\right.
\]
where $s^0=\mathbf{1}_n-\tilde{A}x^0$.

We use the real datasets from LIBSVM~\cite{chang2011libsvm} to test the algorithms. The termination criterion for all algorithms is $200$ epochs. We display the results in~\Cref{table}, from which we can observe that RCS outperforms SubGrad in terms of both the returned objective value and workspace memory consumption (per iteration).


\subsection{Robust phase retrieval problem}
\begin{figure}[t]
	\begin{minipage}[t]{0.24\linewidth}
		\centering
		\includegraphics[width=1.4in]{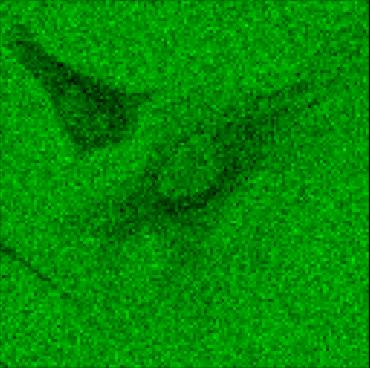}\\
		{ \scriptsize initial image (initial point)}
	\end{minipage}
	\begin{minipage}[t]{0.24\linewidth}
		\centering
		\includegraphics[width=1.4in]{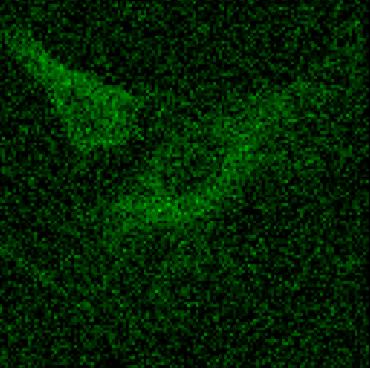}\\
		{\scriptsize SubGrad after $10$ epochs}
	\end{minipage}
	\begin{minipage}[t]{0.24\linewidth}
		\centering
		\includegraphics[width=1.4in]{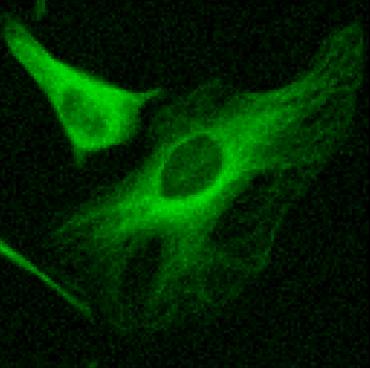}\\
		{\scriptsize RCS after $10$ epochs}
	\end{minipage}
\begin{minipage}[t]{0.24\linewidth}
		\centering
		\includegraphics[width=1.8in]{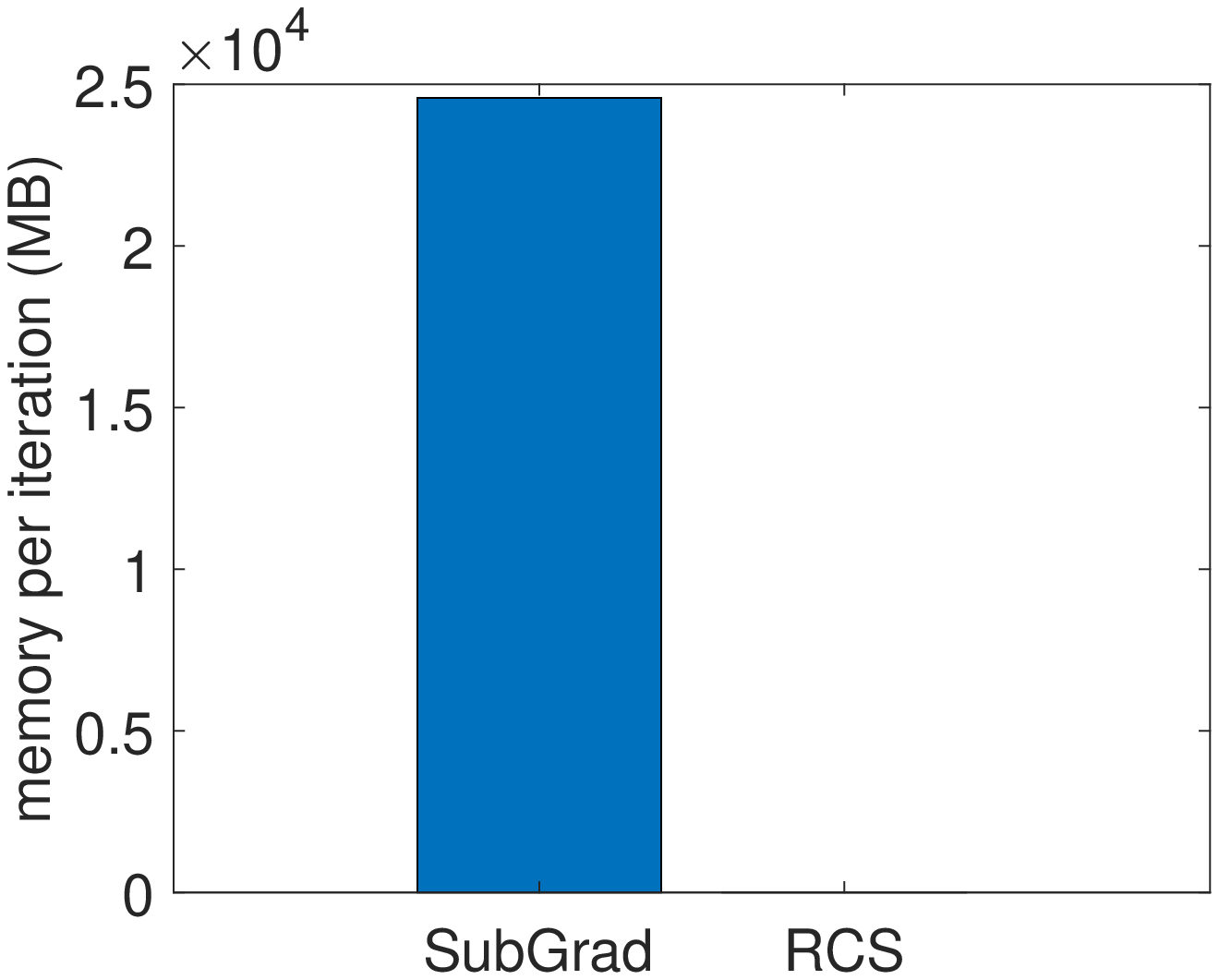}\\
		{\scriptsize  memory per iteration (MB)}
	\end{minipage}
	\caption{\small Experiments on the robust phase retrieval problem with gold nanoparticles correlative images.}\label{fig:01}
\end{figure}

We compare RCS (set $N=d$) with SubGrad on the robust phase retrieval problem \eqref{eq:pr} (i.e., Application 3). We display the problem below for convenience.
\[
\min_{x\in\R^d}\ f(x) := \frac{1}{n}\|(Ax)^{\circ2}-b^{\circ2}\|_1,
\]
where $A=\left(a_1,...,a_n\right)^{\top}\in\R^{n\times d}$ is the data matrix, $b=\left(b_1,...,b_n\right)^{\top}\in\R^n$ is a vector, $\circ$ represents the Hadamard product,  and $\circ2$ is the Hadamard power.

To implement RCS, we introduce the block-wise partition of the columns of the matrix ${A}=\left({A}_{1},{A}_{2},\cdots,{A}_{N}\right)$, where $ A_{i} \in \R^{n\times d_{i}}$ is the $i$-th block.
Then,  the coordinate subgradient $r_{i(k)}^k$ used in RCS can be computed as:
\[
\left\{
\begin{aligned}
	&r_{i(k)}^k = \frac{2}{n}A_{i(k)}^{\top}\left(s^k\circ\sign((s^k)^{\circ2}-b^{\circ2})\right),\\
	&s^{k+1}=s^k+A_{i(k)}(x_{i(k)}^{k+1}-x_{i(k)}^{k}),
\end{aligned}
\right.
\]
where $s^0=Ax^0$.

To test the algorithms, we use the gold nanoparticles correlative images from~\cite{Nature} as the true solution $x^*\in\R^d$. Let $H\in\{-1,1\}^{d\times d}/\sqrt{d}$ be a symmetric normalized Hadamard matrix which satisfies the equation $H^2=I_d$. For some integer $m$, we generate $m$ i.i.d. diagonal sign matrices $S_1,...,S_m\in\diag(\{-1,1\}^d)$ uniformly at random, and define $n=md$ and $A=\left(a_1,...,a_n\right)^{\top}=\left[HS_1,HS_2,...,HS_m\right]^{\top}\in\R^{n\times d}$. Next we generate the outlier vector $\delta=(\delta_1,...,\delta_n)^{\top}\in\R^n$. We randomly select $p_{\mbox{fail}}\cdot n$ locations. Then, we fill each of the selected locations with an i.i.d. $\mathcal{N}(0,1000)$ Gaussian distribution, while the remaining locations are set to $0$. Here, $p_{\mbox{fail}}$ is the ratio of outliers. The measurement vector $(b_1^2,...,b_n^2)^{\top}\in\R^n$ is obtained by $b_i^2=\left\{\begin{array}{lc}(a_i^{\top}x^*)^2&\delta_i=0\\ \delta_i&\delta_i\neq0\end{array}\right.$, for $i=1,...,n$.

In~\Cref{fig:01}, we can see that RCS with just 10 epochs outputs a much clearer retrieved image than the 10 epochs' output of SubGrad. In addition,  RCS requires much less workspace memory than SubGrad per iteration. These properties of RCS are crucial for solving such a huge-scale problem since performing one iteration of SubGrad can be intractable in the scenario where the computational sources are limited. Actually, this is exactly the reason why we do not perform this simulation on the device used for the previous two examples as performing one iteration of SubGrad causes out of memory issue if we use the personal computer.
More epochs are displayed in \Cref{fig:4}.

\section{Conclusion} \label{sec:ccl}
In this paper, we introduced RCS for solving nonsmooth convex and weakly convex composite optimization problems. We considered a general linearly bounded subgradients assumption and conducted a thorough convergence analysis. We also derived a convergence lemma, the relationship between the global metric subregularity properties of a weakly convex function and its Moreau envelope, and the global metric subregularity of the (real-valued) robust phase retrieval problem, which are of independent interests. Moreover, we conducted several experiments to show the superiority of RCS over the subgradient method. Our work reveals a provable and efficient subgradient-type method for a general class of large-scale nonsmooth composite optimization problems.



\appendix
\section{Appendix}
\subsection{Proof of \Cref{lemma:inequality_iter_cvx}}\label{supp:inequality_iter_cvx}

	Taking $x=x^*$ in part (b) of~\Cref{lemma:inequality_iter} and by the convexity of $f$, we can compute
	\begingroup
	\allowdisplaybreaks
	\[
	\begin{aligned}
	\mathbb{E}_{i(k)}\left[\|x^{k+1}-x^*\|^2\right]
		\leq&\|x^{k}-x^*\|^2+\frac{\alpha_k^2}{N}(L_1\|x^k\|+L_2)^2-\frac{2\alpha_k}{N}(f(x^{k})-f^*)\\
		\leq&\|x^{k}-x^*\|^2+\frac{\alpha_k^2}{N}(2L_1^2\|x^k\|^2+2L_2^2)-\frac{2\alpha_k}{N}(f(x^{k})-f^*)\\
		\leq&\|x^{k}-x^*\|^2+\frac{\alpha_k^2}{N}\left(4L_1^2\|x^*-x^k\|^2+4L_1^2\|x^*\|^2+2L_2^2\right)-\frac{2\alpha_k}{N}(f(x^{k})-f^*),
	\end{aligned}
	\]
	\endgroup
	which yields the desired result.

\subsection{Proof of \Cref{cor:rate_qg}}\label{supp:rate_qg}
	By~\Cref{lemma:inequality_iter_cvx_dist}, we have
	\en
		\mathbb{E}_{\mathcal{F}_{k}} [\dist^2(x^{k+1},\setX^*)]-\mathbb{E}_{\mathcal{F}_{k-1}}[\dist^2(x^{k},\setX^*)]
		\leq\frac{C_1}{N}\alpha_k^2-\frac{2\alpha_k}{N}\mathbb{E}_{\mathcal{F}_{k-1}}(f(x^k)-f^*).	
	\een
	Since $f$ satisfy the global quadratic growth condition with $\kappa_3$, we obtain
	\be\label{eq:rate_qg}
		\mathbb{E}_{\mathcal{F}_{k}} [\dist^2(x^{k+1},\setX^*)]-\mathbb{E}_{\mathcal{F}_{k-1}}[\dist^2(x^{k},\setX^*)]
		\leq\frac{C_1}{N}\alpha_k^2-\frac{2\alpha_k}{N\kappa_3}\mathbb{E}_{\mathcal{F}_{k-1}}[\dist^2(x^k,\setX^*)].
	\ee
	Let $\alpha_k=\frac{N\kappa_3}{k+1}$. \eqref{eq:rate_qg} reduces to
	\[
\mathbb{E}_{\mathcal{F}_{k}} [\dist^2(x^{k+1},\setX^*)]-\mathbb{E}_{\mathcal{F}_{k-1}}[\dist^2(x^{k},\setX^*)]\\
		\leq\frac{NC_1\kappa_3^2}{(k+1)^2}-\frac{2}{k+1}\mathbb{E}_{\mathcal{F}_{k-1}}[\dist^2(x^k,\setX^*)].
	\]
	Multiplying both sides of the above inequality by $(k+1)^2$ gives
	\[
(k+1)^2\mathbb{E}_{\mathcal{F}_{k}} [\dist^2(x^{k+1},\setX^*)]
		-\left[(k+1)^2-2k-2\right]\mathbb{E}_{\mathcal{F}_{k-1}}[\dist^2(x^{k},\setX^*)]\leq NC_1\kappa_3^2.
	\]
	Note the fact that $(k+1)^2-2k-2=k^2-1\leq k^2$.  Then, we further have
	\[
(k+1)^2\mathbb{E}_{\mathcal{F}_{k}} [\dist^2(x^{k+1},\setX^*)]-k^2\mathbb{E}_{\mathcal{F}_{k-1}}[\dist^2(x^{k},\setX^*)]\leq NC_1\kappa_3^2.
	\]
	Therefore,
	\[
	\begin{aligned}
		(k+1)^2\mathbb{E}_{\mathcal{F}_{k}} [\dist^2(x^{k+1},\setX^*)]	=&\sum_{j=0}^k\bigg{[}(j+1)^2\mathbb{E}_{\mathcal{F}_{j}} [\dist^2(x^{j+1},\setX^*)]-j^2\mathbb{E}_{\mathcal{F}_{j-1}}[\dist^2(x^{j},\setX^*)]\bigg{]}\\
		\leq& NC_1\kappa_3^2(k+1),
	\end{aligned}
	\]
	which yields
	\[
	\mathbb{E}_{\mathcal{F}_{k}} [\dist^2(x^{k+1},\setX^*)]\leq\frac{NC_1\kappa_3^2}{k+1}.
	\]

\subsection{Proof of~\Cref{fact:pr_boundcritical}}\label{appen:pr_boundcritical}
The robust phase retrieval problem~\eqref{eq:pr} can be rewritten as
\be\label{eq:pr_matrix}
\min_{x\in\R^d}\quad f(x)=\frac{1}{n}\|(Ax)^{\circ2}-b^{\circ2}\|_1=\frac{1}{n}\sum_{i=1}^n|\left(a_i^{\top}x\right)^{2}-b_i^2|,
\ee
where $A=\left(a_1,...,a_n\right)^{\top}\in\R^{n\times d}$, $b=(b_1,...,b_n)^{\top}\in\R^{n}$ and $\circ2$ is the Hadamard power. And the subgradient of robust phase retrieval problem is
\[
	\partial f(x)=\frac{2}{n}A^{\top}\left(Ax\circ\xi\right)=\frac{2}{n}\sum_{i=1}^na_i\left(a_i^{\top}x\cdot\xi_i\right)\quad\mbox{with}\quad\xi=\left(\xi_1,...,\xi_n\right)^{\top},
\]
$\xi_i\in\left\{\begin{aligned}\{-1\}&\qquad(a_i^{\top}x)^2-b_i^2<0\\ [-1,1]&\qquad(a_i^{\top}x)^2-b_i^2=0\\ \{1\}&\qquad(a_i^{\top}x)^2-b_i^2>0\end{aligned}\right.$, and $\circ$ is the Hadamard product.

Intuitively, for any given $x\in\R^d$, we can divided the index of datasets $\{(a_1,b_1),...,(a_n,b_n)\}$ into three subsets:
\be\label{eq:pr_bd_0}
\begin{aligned}
	\setI_1&=\{i\in\{1,...,n\}|(a_i^{\top}x)^2=b_i^2\},\\
	\setI_2&=\{i\in\{1,...,n\}|(a_i^{\top}x)^2>b_i^2\},\\
	\setI_3&=\{i\in\{1,...,n\}|(a_i^{\top}x)^2<b_i^2\}.
\end{aligned}
\ee
We can easily find the fact that $\setI_1\cap\setI_2=\emptyset$, $\setI_2\cap\setI_3=\emptyset$, $\setI_1\cap\setI_3=\emptyset$, and $\setI_1\cup\setI_2\cup\setI_3=\{1,...,n\}$. For given $x\in\R^d$, the subdifferential of robust phase retrieval problem can be rewritten as
\be\label{eq:pr_bd_1}
\partial f(x)=\frac{2}{n}\left[\sum_{i\in\setI_1} a_i\left(a_i^{\top}x\cdot[-1,1]\right)+\sum_{i\in\setI_2} a_ia_i^{\top}x-\sum_{i\in\setI_3} a_ia_i^{\top}x\right].
\ee
If $x=\bar{x}\in\overline{\setX}$,~\eqref{eq:pr_bd_1} yields that
\be\label{eq:pr_bd_2}
\sum_{i\in\setI_2} a_ia_i^{\top}\bar{x}=\sum_{i\in\setI_3} a_ia_i^{\top}\bar{x}-\sum_{i\in\setI_1} a_i\left(a_i^{\top}\bar{x}\cdot\delta_i\right),
\ee
with $\delta_i\in[-1,1]$, $i=1,...,n$. By the full column rank of $A$, we denote the positive definite matrix $Q=A^{\top}A$, and $\sigma_{min}(Q^{\top}Q)$ be the minimum eigenvalue of matrix $Q^{\top}Q$. Then we can derive that
\be\label{eq:pr_bd_3}
\begin{aligned}
\sqrt{\sigma_{\min}(Q^{\top}Q)}\|\bar{x}\|
	\leq&\|A^{\top}A\bar{x}\|\\
	=&\|\sum_{i=1}^n a_ia_i^{\top}\bar{x}\|\\
	=&\|\sum_{i\in\setI_1} a_ia_i^{\top}\bar{x}+\sum_{i\in\setI_2} a_ia_i^{\top}\bar{x}+\sum_{i\in\setI_3} a_ia_i^{\top}\bar{x}\|\\
	=&\|2\sum_{i\in\setI_3} a_ia_i^{\top}\bar{x}+\sum_{i\in\setI_1} a_ia_i^{\top}\bar{x}-\sum_{i\in\setI_1} a_i\left(a_i^{\top}\bar{x}\cdot\delta_i\right)\|\quad\mbox{(by~\eqref{eq:pr_bd_2})}\\
	\leq&2\sum_{i\in\setI_3}\|a_i\|\cdot|a_i^{\top}\bar{x}|+2\sum_{i\in\setI_1}\|a_i\|\cdot|a_i^{\top}\bar{x}|.
\end{aligned}
\ee
By~\eqref{eq:pr_bd_0}, we have $|a_i^{\top}\bar{x}|=|b_i|$, $\forall i\in\setI_1$ and $|a_i^{\top}\bar{x}|<|b_i|$, $\forall i\in\setI_3$. Then~\eqref{eq:pr_bd_3} yields that
\[
\|\bar{x}\|\leq\frac{2\sum_{i\in\setI_3}\|a_i\|\cdot|b_i|+2\sum_{i\in\setI_1}\|a_i\|\cdot|b_i|}{\sqrt{\sigma_{\min}(Q^{\top}Q)}}.
\]
Therefore, under the assumption that $A$ is full column rank, the critical points set of the robust phase retrieval problem~\eqref{eq:pr} is bounded, i.e., $\max_{\bar{x}\in\overline{\setX}}\|\bar{x}\|\leq\sB_2=\frac{2\sum_{i\in\setI_3}\|a_i\|\cdot|b_i|+2\sum_{i\in\setI_1}\|a_i\|\cdot|b_i|}{\sqrt{\sigma_{\min}(Q^{\top}Q)}}$.

\subsection{Proof of \Cref{cor:asymptoticrate_wcvx}}\label{supp:asymptoticrate_wcvx}

Note that the step sizes $\alpha_k = \frac{\Delta}{\sqrt{k+1}\log(k+2)}$. The two conditions in~\eqref{eq:step size cond} are automatically satisfied since
	\[
	\sum_{j=0}^k\alpha_j^2 \leq  \alpha_0^2 + \int_{0}^{k} \frac{\Delta^2}{(t+1)\log^2(t+2)} dt \leq \frac{2\Delta^2}{(\log 2)^2},
	\]
	and
	\[
	\sum_{j=0}^{k}\alpha_j \geq \sum_{j=0}^{k} \frac{\Delta}{\sqrt{k+1}\log(k+2)} = \frac{\Delta \sqrt{k+1}}{\log(k+2)}.
	\]
	We now prove this corollary by contradiction. Let us assume that the statement does not hold. Then, we have
	\[
	\mathbb{P}\left(\left\{ w\in \Omega: \liminf_{k\rightarrow\infty} \ \sqrt{k+1}\|\nabla f_{\lambda}(x^k(w))\|^2\geq \delta \right\}\right) >0,
	\]
	for some $\delta>0$. Then, there exists a sufficiently large $\overline{k}$ such that
	\[
	\mathbb{P}\left(\left\{ w\in \Omega: \sqrt{k+1} \ \|\nabla f_{\lambda}(x^k(w))\|^2 \geq \delta \right\}\right)>0
	\]
	for all $k\geq \overline{k}$. Hence, we have
	\be\label{eq:contradict 3}
	\mathbb{P}\left(\left\{w\in \Omega:  \sum_{k=\overline{k}}^{\infty}\frac{1}{\sqrt{k+1}\log(k+2)}\|\nabla f_{\lambda}(x^k(w))\|^2 \geq\sum_{k=\overline{k}}^{\infty}\frac{\delta}{(k+1)\log(k+2)}\right\}\right)>0.
	\ee
	By the integral comparison test, we have  $\sum_{k=\overline{k}}^{\infty}\frac{1}{(k+1)\log(k+2)}  \geq \int_{\overline{k} + 2}^{\infty} \frac{1}{t\log(t)} dt = \infty$.
	Note that $\alpha_k = \frac{\Delta}{\sqrt{k+1}\log(k+2)}$. Therefore,  \eqref{eq:contradict 3} is a contradiction to \eqref{eq:as_rate_contradiction}.

\subsection{RCS for M-estimators problem with MCP-loss and $\ell_1$-penalty, and more experiments}\label{num:MCPMestimator}
In this subsection, we consider the M-estimators problem \eqref{eq:M-esti} with MCP-loss and $\ell_1$-penalty, i.e.,
\[
\min\limits_{x\in\R^d} \ f(x):=\frac{1}{n}\sum_{i=1}^n\phi_{p_1}(a_i^{\top}x-b_i)+p_2\|x\|_1,
\]
where $a_i\in\R^d$ is a vector,  $b_i\in\R$ is a scalar, $p_2>0$ is the regularization parameter, and $\phi_{p_1}(z)=\left\{\begin{array}{ll}|z|-\frac{z^2}{2p_1}&|z|\leq p_1\\ \frac{1}{2}p_1&|z|> p_1\end{array}\right.$ is the MCP-loss.

Let $A=(a_1,...,a_n)^{\top}\in\R^{n\times d}$ be a matrix and $b=(b_1,...,b_n)^{\top}\in\R^n$ be a vector.  In order to implement RCS, we introduce the block-wise partition of the columns of the data matrix $A=(A_{1},A_{2},\cdots,A_{N})$, where $A_{i} \in \R^{n\times d_{i}}$ is the $i$-th block. For this problem, the coordinate subgradient used in RCS can be calculated as:
\[
\left\{
\begin{aligned}
	&r_{i(k)}^k = \frac{1}{n}A_{i(k)}^{\top}u^k+p_2\sign(x_{i(k)}^k),\\
	&s^{k+1}=s^k+A_{i(k)}(x_{i(k)}^{k+1}-x_{i(k)}^{k}),
\end{aligned}
\right.
\]
where $u^k \in \partial\Phi(s^k)=\left(\partial\phi_{p_1}(s_1^k),...,\partial\phi_{p_1}(s_n^k)\right)^{\top}$ with $\partial\phi_{p_1}(z)=\left\{\begin{array}{ll}\sign(z)-\frac{z}{p_1}&|z|\leq p_1\\ 0&|z|> p_1\end{array}\right.$ and $s^0=Ax^0-b$.

We conduct our simulation based on the synthetic data generated in \Cref{sec:M-esti}. In~\Cref{fig:2} and~\Cref{fig:2-1}, we show the evolution of $f(x^k)-f^*$ and $\dist(x^k-x^*)$ versus epoch counts with different number of blocks $N$.  Similar to the conclusions draw from \Cref{fig:00}, we can observe that RCS with larger $N$ converges faster than RCS with smaller $N$ and SubGrad. In addition, larger $N$ requires less workspace memory than small $N$ and SubGrad per iteration. Finally, RCS can effectively recover the sparse coefficients.
\begin{figure}
	\begin{minipage}[t]{0.24\linewidth}
		\centering
		\includegraphics[width=1.7in]{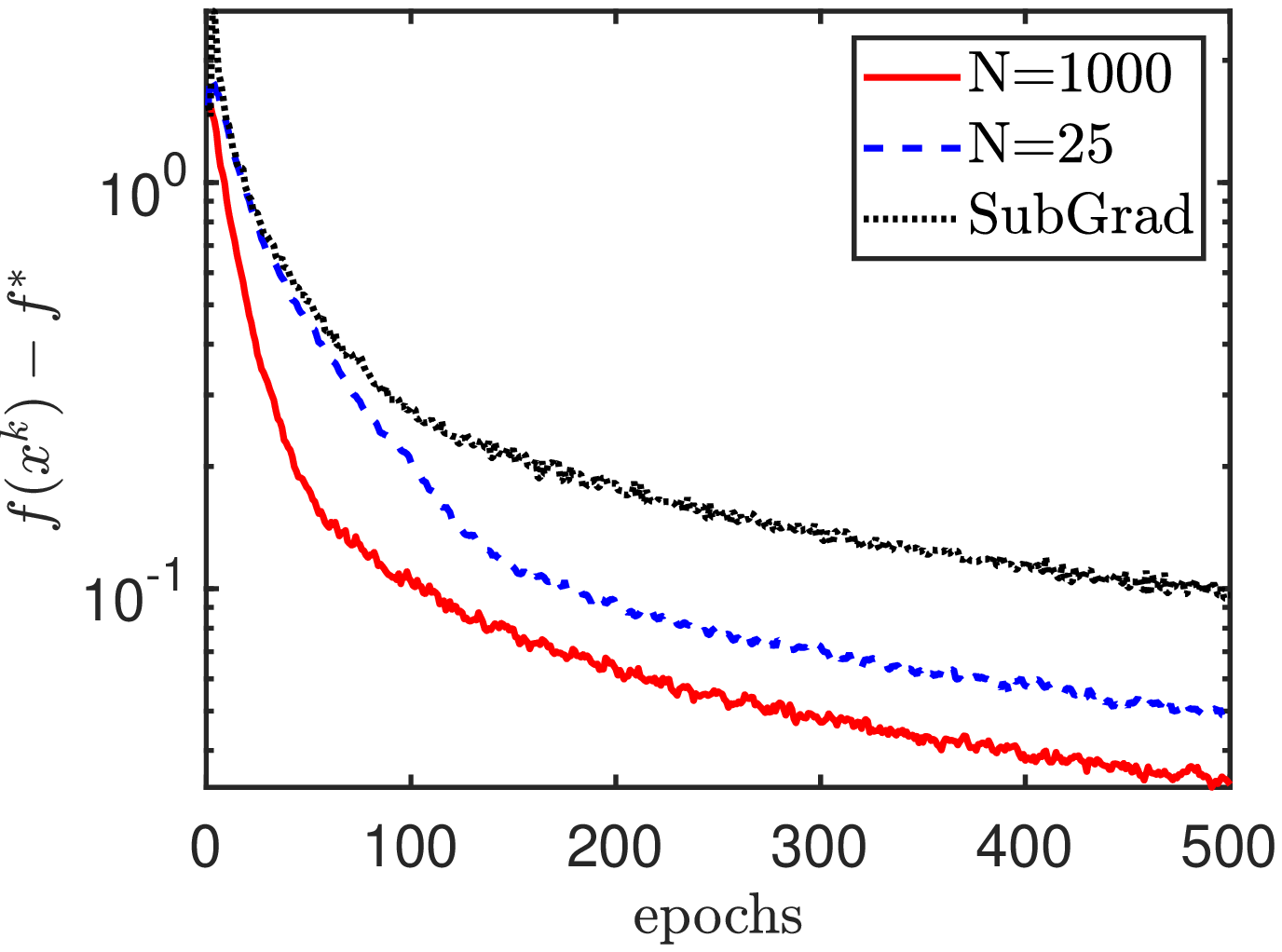}
	\end{minipage}
    \begin{minipage}[t]{0.24\linewidth}
		\centering
		\includegraphics[width=1.7in]{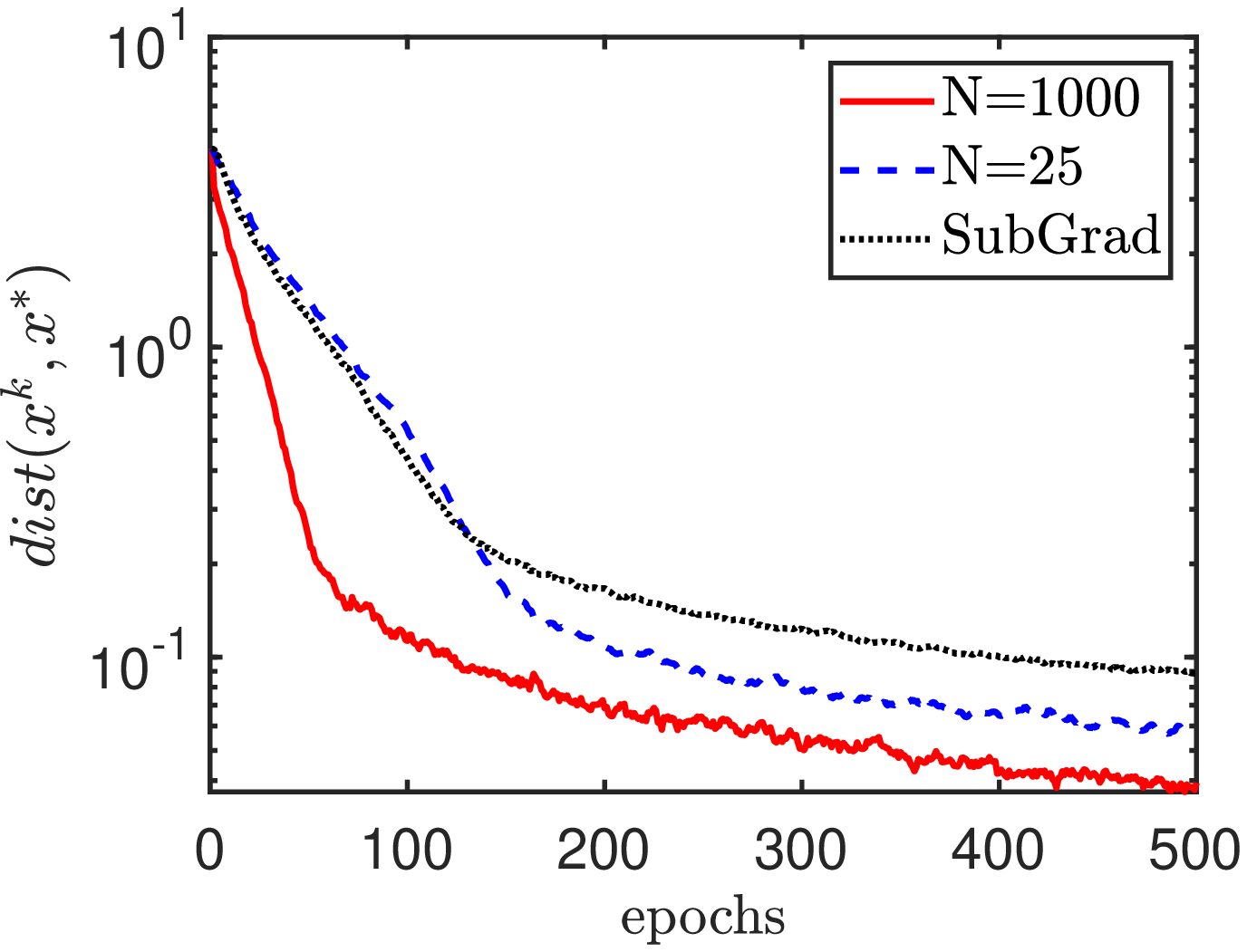}
	\end{minipage}
	\begin{minipage}[t]{0.24\linewidth}
		\centering
		\includegraphics[width=1.7in]{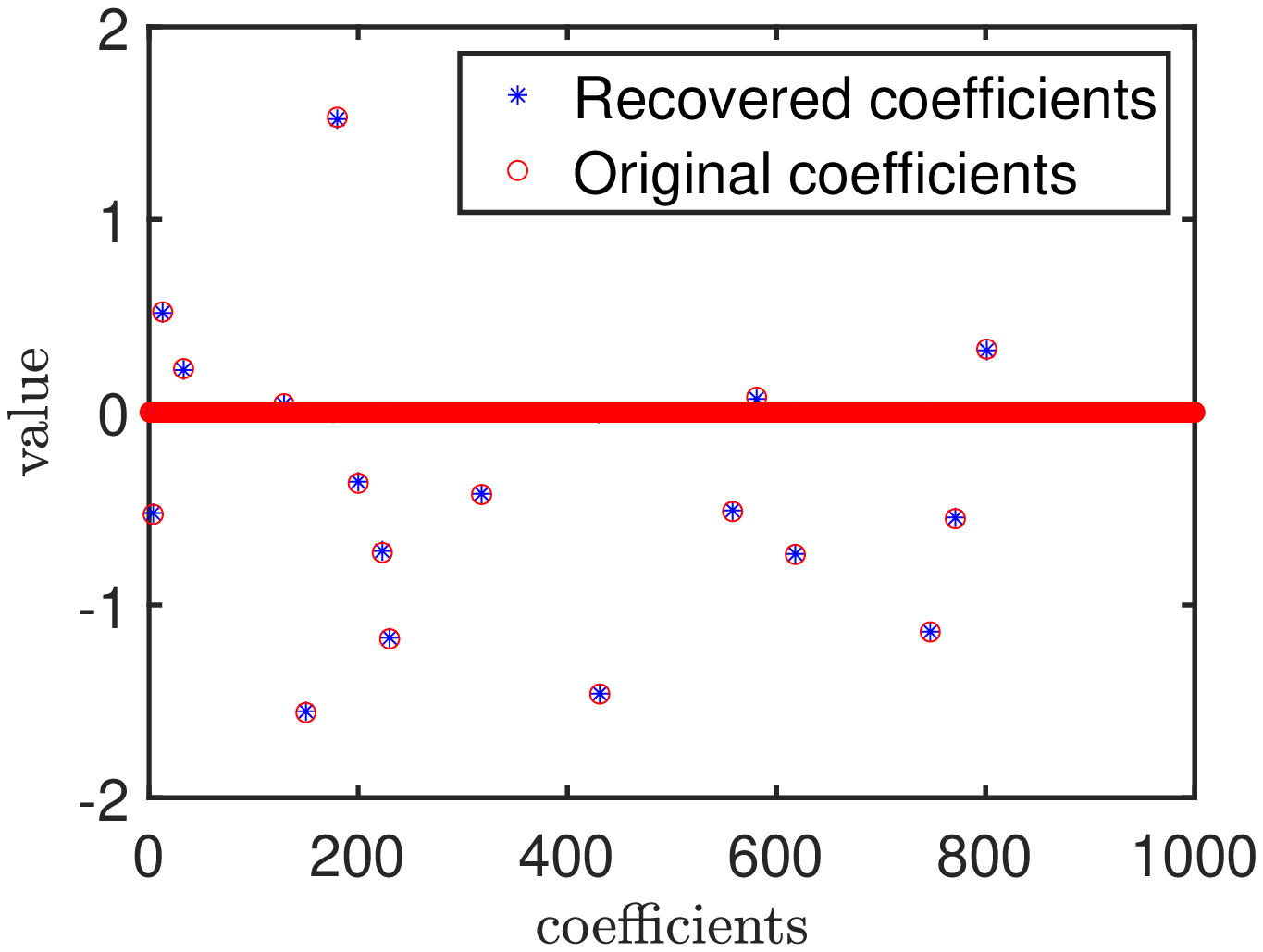}
	\end{minipage}
	\begin{minipage}[t]{0.24\linewidth}
		\centering
		\includegraphics[width=1.7in]{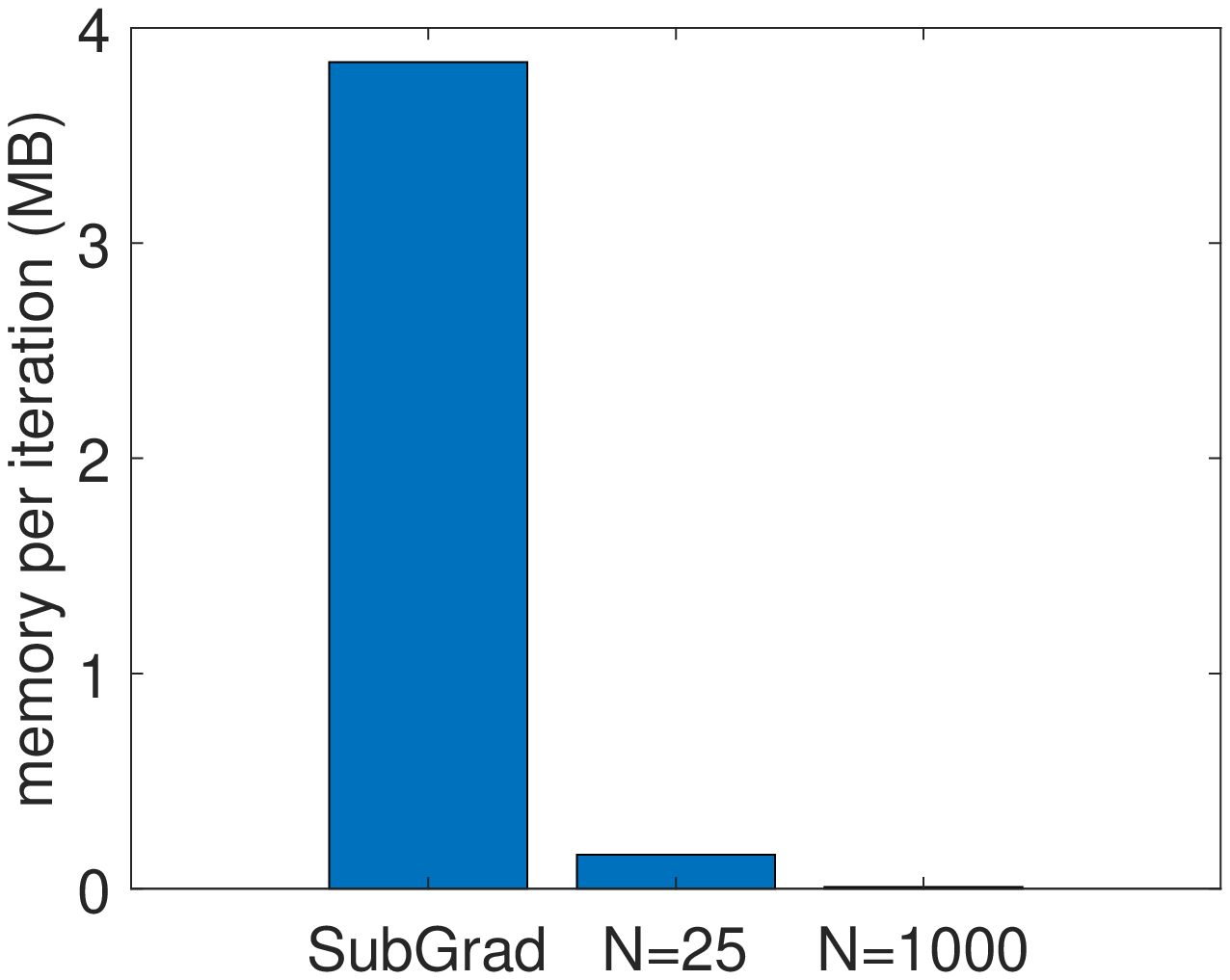}
	\end{minipage}
	\caption{Experiments on Robust M-estimators problem with MCP-loss and $\ell_1$-penalty. {\tt Top Left:} $f(x^k)-f^*$ versus epoch counts for different choices of $N$ (i.e., the number of blocks in RCS); {\tt Top Right:} $\dist(x^k,x^*)$ versus epoch counts for different choices of $N$; {\tt Bottom Left:} Coefficients recovery by RCS with $N=d$; {\tt Bottom Right:} Comparison on workspace memory consumption per iteration.  Here, $n = 500$, $d=1000$, $s = 20$, and $p_{\mbox{fail}} = 0.25$.\label{fig:2}}
\end{figure}
\begin{figure}
	\begin{minipage}[t]{0.24\linewidth}
		\centering
		\includegraphics[width=1.7in]{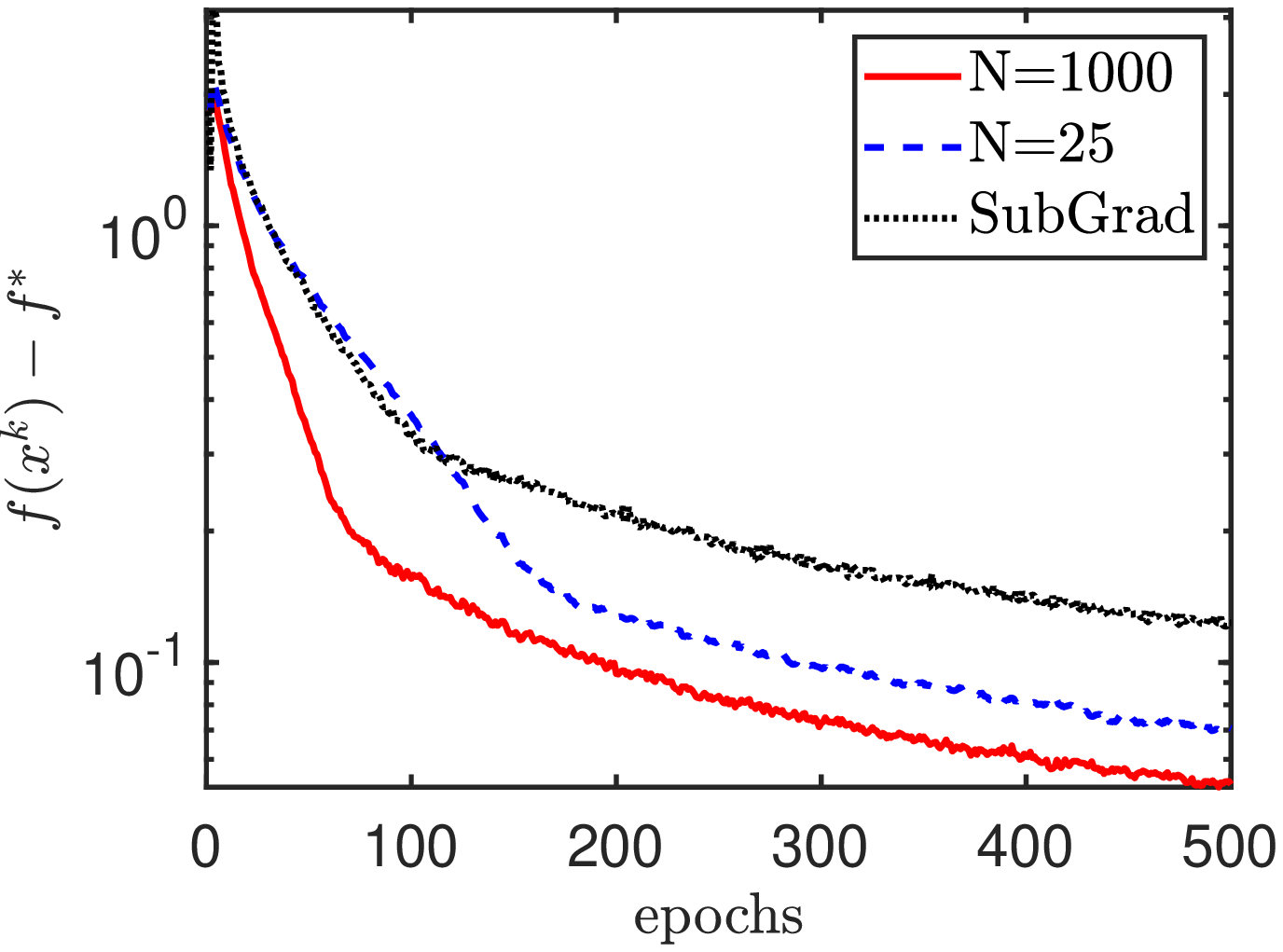}
	\end{minipage}
    \begin{minipage}[t]{0.24\linewidth}
		\centering
		\includegraphics[width=1.7in]{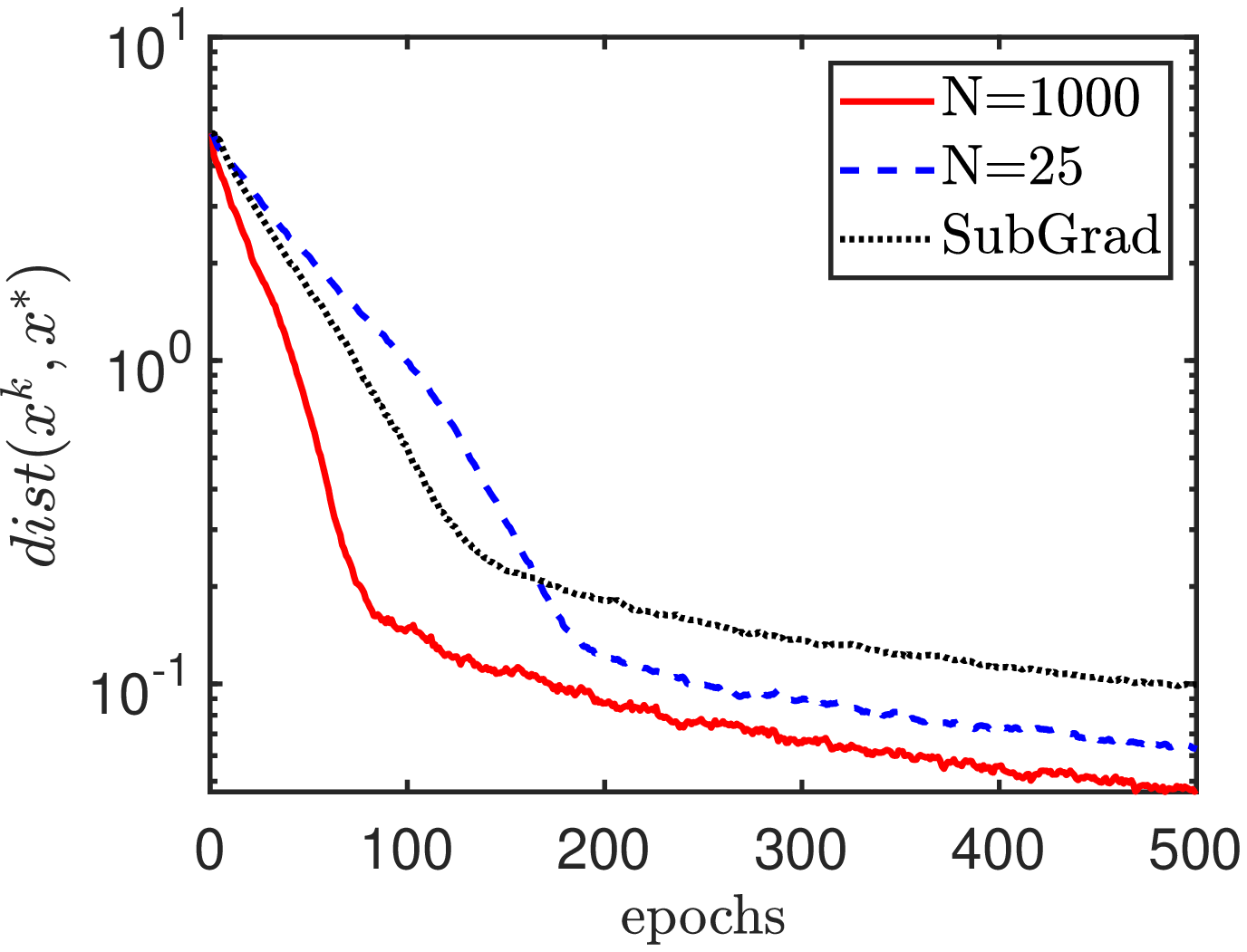}
	\end{minipage}
    \begin{minipage}[t]{0.24\linewidth}
		\centering
		\includegraphics[width=1.7in]{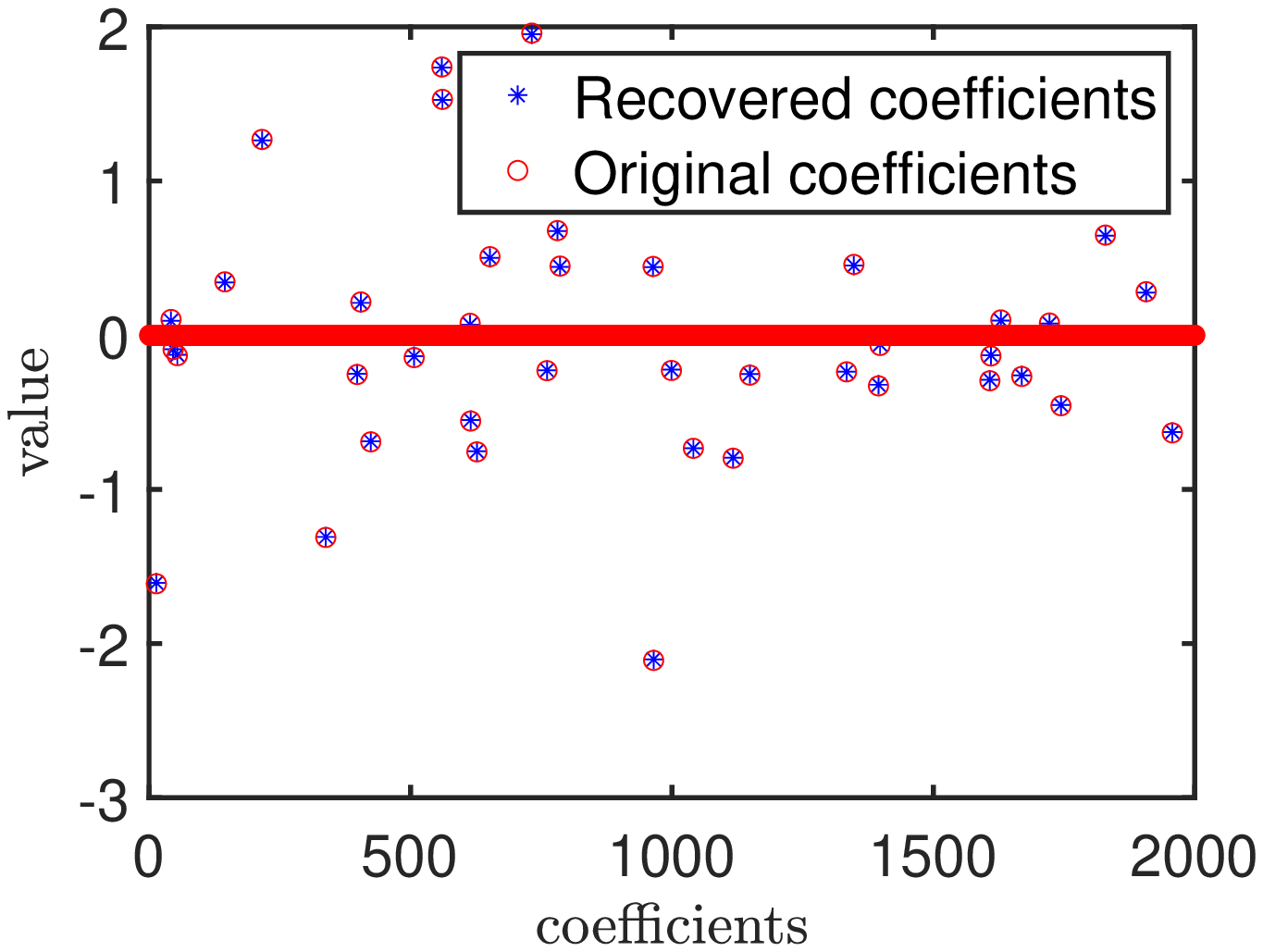}
	\end{minipage}
	\begin{minipage}[t]{0.24\linewidth}
		\centering
		\includegraphics[width=1.7in]{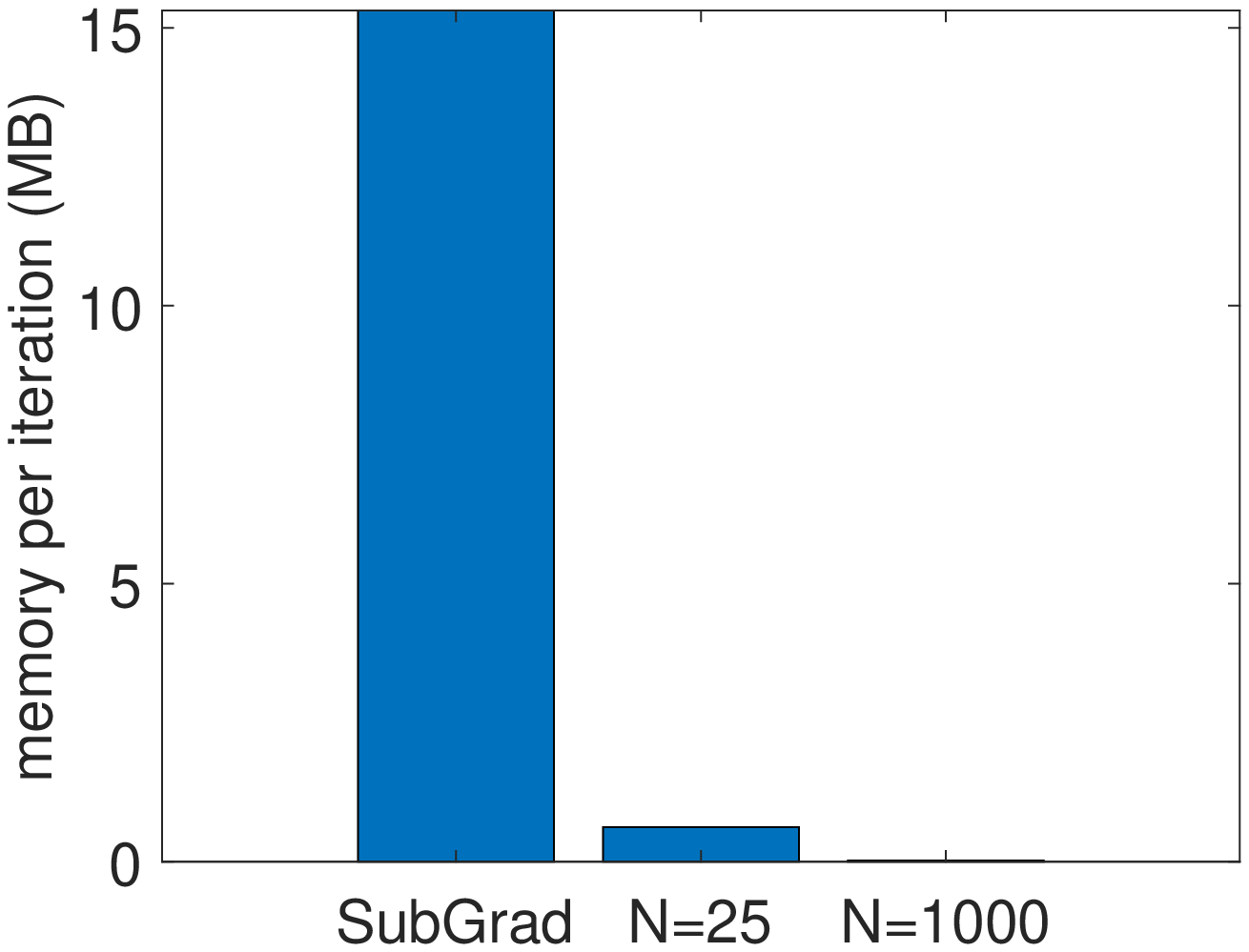}
	\end{minipage}
	\caption{Experiments on Robust M-estimators problem with MCP-loss and $\ell_1$-penalty. {\tt Top Left:} $f(x^k)-f^*$ versus epoch counts for different choices of $N$ (i.e., the number of blocks in RCS); {\tt Top Right:} $\dist(x^k,x^*)$ versus epoch counts for different choices of $N$; {\tt Bottom Left:} Coefficients recovery by RCS with $N=d$; {\tt Bottom Right:} Comparison on workspace memory consumption per iteration. Here, $n = 1000$, $d=2000$, $s = 40$, and $p_{\mbox{fail}} = 0.25$. \label{fig:2-1}}
\end{figure}

\begin{figure}[htbp]
	\centering
	\begin{minipage}[t]{0.24\linewidth}
		\centering
		\includegraphics[width=1.4in]{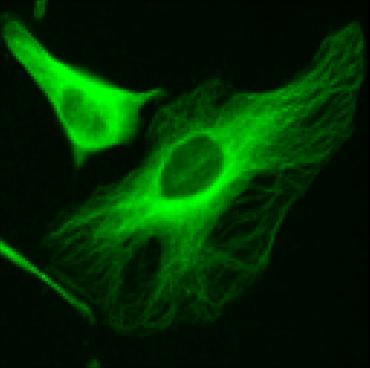}
		{\scriptsize original image}
	\end{minipage}
	\begin{minipage}[t]{0.24\linewidth}
		\centering
		\includegraphics[width=1.4in]{result_v0RCS_random.jpg}
		{\scriptsize initial image (initial point)}
	\end{minipage}\\
	\begin{minipage}[t]{0.24\linewidth}
		\centering
		\includegraphics[width=1.4in]{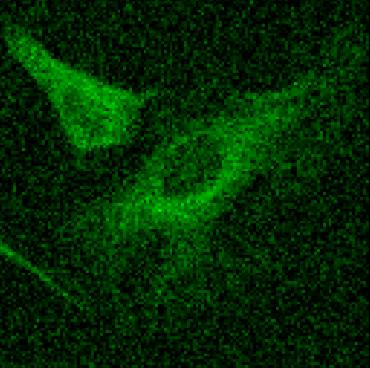}
		{\scriptsize RCS after $5$ epochs}
	\end{minipage}
	\begin{minipage}[t]{0.24\linewidth}
		\centering
		\includegraphics[width=1.4in]{result_v2RCS_random.jpg}
		{\scriptsize RCS after $10$ epochs}
	\end{minipage}
	\begin{minipage}[t]{0.24\linewidth}
		\centering
		\includegraphics[width=1.4in]{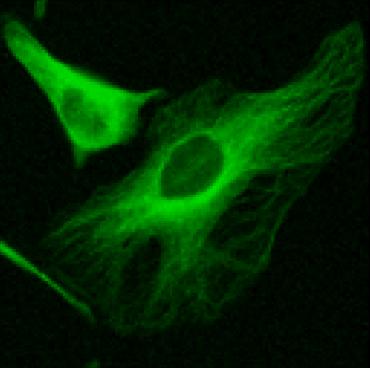}
		{\scriptsize RCS after $15$ epochs}
	\end{minipage}
	\begin{minipage}[t]{0.24\linewidth}
		\centering
		\includegraphics[width=1.4in]{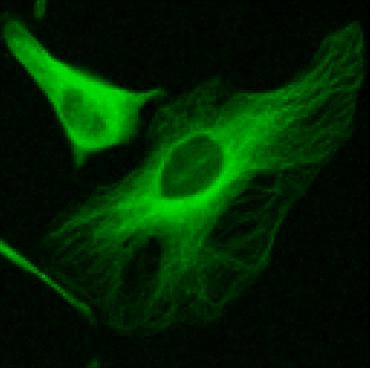}
		{\scriptsize RCS after $20$ epochs}
	\end{minipage}\\
	\begin{minipage}[t]{0.24\linewidth}
		\centering
		\includegraphics[width=1.4in]{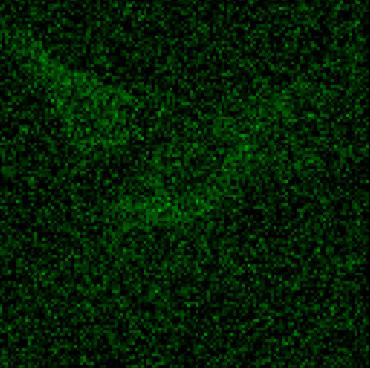}
		{\scriptsize SubGrad after $5$ epochs}
	\end{minipage}
	\begin{minipage}[t]{0.24\linewidth}
		\centering
		\includegraphics[width=1.4in]{result_v2subgradient.jpg}
		{\scriptsize SubGrad after $10$ epochs}
	\end{minipage}
	\begin{minipage}[t]{0.24\linewidth}
		\centering
		\includegraphics[width=1.4in]{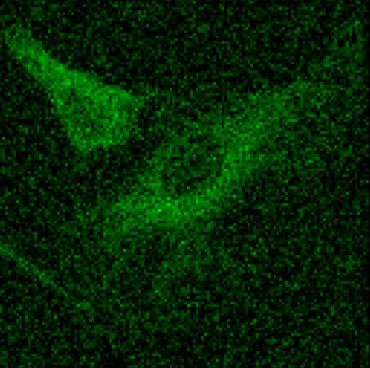}
		{\scriptsize SubGrad after $15$ epochs}
	\end{minipage}
	\begin{minipage}[t]{0.24\linewidth}
		\centering
		\includegraphics[width=1.4in]{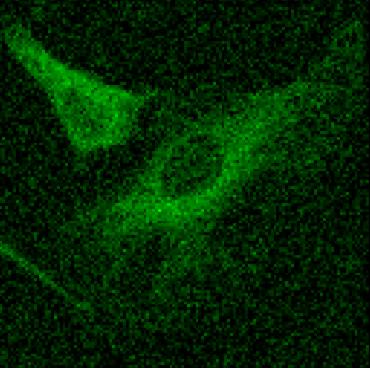}
		{\scriptsize SubGrad after $20$ epochs}
	\end{minipage}
	\caption{ A supplemental to \Cref{fig:01} by showing more epochs.}\label{fig:4}
\end{figure}

\bibliographystyle{IEEEtran}
\bibliography{paper}
\end{document}